\documentclass[11pt]{amsart}
\usepackage[a4paper, top=1in, bottom=1.25in, left=1.25in, right=1.25in]{geometry}
\usepackage{amsmath, amssymb,verbatim, amsfonts, mathtools, xcolor,enumerate}
\usepackage{hyperref,marginnote}
\hypersetup{
    colorlinks = true,
    citecolor= {teal},
    linkcolor = {teal},
}

\newtheorem{theorem}{Theorem}[section]

\newtheorem{proposition}[theorem]{Proposition}
\newtheorem{corollary}[theorem]{Corollary}
\newtheorem{lemma}[theorem]{Lemma}
\theoremstyle{remark}
\newtheorem{example}[theorem]{Example}

\newtheorem{remark}[theorem]{Remark}

\newcommand{\cH}{\mathcal{H}}

\newcommand{\cB}{\mathcal{B}}

\newcommand{\cT}{\mathcal{T}}

\newcommand{\cM}{\mathcal{M}}

\newcommand{\eps}{\varepsilon}

\newcommand{\real}{\operatorname{Re}}

\newcommand{\scp}[1]{\left\langle#1\right\rangle}

\newcommand{\Pp}{\mathbf{P}_{+}}
\newcommand{\Pm}{\mathbf{P}_{-}}
\newcommand{\Comp}{\mathbb C}
\newcommand{\norm}[1]{\left\| #1 \right\|}

\DeclareMathOperator{\diag}{diag}
\DeclareMathOperator{\e}{e}
\DeclareMathOperator{\RE}{Re}
\DeclareMathOperator{\ii}{i}
\title[Contractive realization theory]{Contractive realization theory for the annulus and other intersections of discs on the Riemann sphere   
}
\author[R. Baran]{Radomił Baran}
\address{Faculty of Mathematics and Computer Science, Jagiellonian University, ul. \L ojasiewicza 6, 30-348 Krak\'ow, Poland,}
\email{radomil.baran@doctoral.uj.edu.pl}

\author[P. Pikul]{Piotr Pikul}
\address{Faculty of Mathematics and Computer Science, Jagiellonian University, ul. \L ojasiewicza 6, 30-348 Krak\'ow, Poland,}
\email{piotr.pikul@uj.edu.pl}

\author[H. J. Woerdeman]{Hugo J. Woerdeman}
\address{H. J. Woerdeman, Department of Mathematics, Drexel University, Philadelphia, PA 19104, USA}
\email{hugo@math.drexel.edu}

\author[M. Wojtylak]{Micha\l{} Wojtylak}
\address{Faculty of Mathematics and Computer Science, Jagiellonian University, ul. \L ojasiewicza 6, 30-348 Krak\'ow, Poland,}
\email{michal.wojtylak@uj.edu.pl}

\begin{document}


\maketitle

\begin{abstract}
We develop contractive finite dimensional realizations for rational matrix functions of one variable on domains that are not simply connected, such as the annulus. The proof uses multivariable contractive realization results as well as abstract operator algebra techniques. Other results include  new  bounds for the Bohr radius of the bidisk and the annulus. 
\end{abstract}

\section{Introduction}

A classical result due to Arov \cite{Arov} states that a rational matrix function $F(z)$ (of size $k\times l$) that takes on contractive values for $z\in {\mathbb D}=\{ z \in{\mathbb C} : |z|<1\}$ has a {\em contractive finite dimensional realization}; that is, there exists a contractive block matrix ${\tiny \begin{bmatrix} A & B \\ C & D \end{bmatrix}}\in\Comp^{d+k,d+l}$ so that 
\begin{equation}\label{real} F(z) = D +zC (I -zA)^{-1} B   ,\quad z\in {\mathbb D}. \end{equation}
Earlier results, that go back to \cite{dBR1,dBR2} (\cite{BallCohen} provides an overview), concern {\em analytic} operator-valued functions that take on contractive values on the open unit disk; in this case, a realization \eqref{real} exists where now $A,B,C,$ and $D$ are Hilbert space operators.
The theory of realizations is of importance in control and systems theory and in interpolation problems, and it provides a useful tool in operator theory in general; see, e.g., the monographs \cite{BGK,BGR}. 

In \cite{Agler} multivariable analogs were considered, where now the realization takes the form
\begin{equation}\label{real2} F(z) = D +C{\bf P}_-(z) (I -A{\bf P}_- (z))^{-1} B,
\quad {\bf P}_-(z)= \bigoplus\nolimits_{j=1}^d z_jI_{{\mathcal H}_j}, \  z  \in {\mathbb D}^d. \end{equation}
Agler \cite{Agler} recognized that an immediate generalization of the single variable result does not work. Indeed, a calculation (see, e.g, the proof of Proposition \ref{Jconminreal} below) shows that \eqref{real2} implies that
\begin{equation}\label{ineq} I - F(z)^*F(z) \ge C^* (I -A{\bf P}_- (z))^{*-1} (I - {\bf P}_-(z)^* {\bf P}_-(z) ) (I -A{\bf P}_- (z))^{-1} C , \end{equation}
where $\ge$ is the Loewner order; i.e., $P\ge Q$ iff $P-Q$ is positive semidefinite. This leads to the observation that if $z  \in {\mathbb D}^d$ is replaced by a tuple ${\bf T}=(T_1, \ldots , T_d)$ of commuting Hilbert space strict contractions (using the  Riesz functional calculus), we obtain that $\| F({\bf T}) \| \le 1$. As a consequence, the so-called Agler norm was introduced:
$$
\| F (\cdot) \|_{\mathcal T^{\circ}_{{\mathbb D}^d}} := \sup \{ \| F({\bf T}) \| : {\bf T} \hbox{\ is a tuple of commuting strict contractions} \}. 
$$
Agler's seminal result \cite{Agler} is that $ \| F (\cdot) \|_{\mathcal T^{\circ}_{{\mathbb D}^d}} \le 1$ if and only if $F(z)$ has a contractive realization \eqref{real2}. In the case of two variables, the Agler norm and the supremum norm coincide, due to a result by And\^o \cite{Ando}. This also allows for a finite dimensional result (see \cite{Kummert, Knese2021}), where a rational matrix function $F(z_1,z_2)$ takes on contractive values for $(z_1,z_2)\in {\mathbb D}^2$ if and only if there exists a contractive block matrix ${\tiny \begin{bmatrix} A & B \\ C & D \end{bmatrix}}$ so that
\begin{equation}\label{real2a} F(z) = D +C{\bf P}_-(z) (I -A{\bf P}_- (z))^{-1} B   , \ \ {\bf P}_-(z_1,z_2) = z_1I_{n_1}\oplus z_2 I_{n_2},  z  \in {\mathbb D}^2. \end{equation}
The finite dimensional realization theory was used to obtain determinantal representation result for stable polynomials; see \cite{Kummert,Woerdeman2012, grinshpan2016matrixvaluedhermitianpositivstellensatzlurking, GKVW2016, JW}. 

The results of Agler were generalized to many other domains (see, e.g. \cite{AT, BB}). Also, for the rational matrix function case, with finite dimensional realizations, generalizations were derived in \cite{GKW2013} and \cite{VW}. The latter paper also considers realizations of the form 
\begin{equation}\label{real3} F(z) = D +C{\bf P}_-(z) ({\bf P}_+(z) -A{\bf P}_- (z))^{-1} B   ,  \end{equation} where ${\bf P}_\pm$ are matrix valued polynomials that describe the domain of interest via the inequality ${\bf P}_+(z)^* {\bf P}_+(z) - {\bf P}_-(z)^*{\bf P}_-(z) >0$. In all the results mentioned above the domains $\Omega$ are polynomially convex (which, in the one-variable case, means that the complement ${\mathbb C}\setminus \Omega$ is connected). 

In this paper we go beyond the setting of polynomial convexity, and allow the domain $\Omega \subseteq {\mathbb C}$ to have holes.
 Namely, we will consider $\Omega$ to be a bounded intersection of discs on the Riemann sphere, a case for which the polynomial matrices $\mathbf P_{\pm}(z)$ arise naturally; see Section~\ref{s3} for details.
A first example to think of is the annulus $\Omega= \{ z \in {\mathbb C} : 0<r <|z|<R \},$ where  ${\bf P}_\pm$ take on the form
$$
{\bf P}_+(z) = RI_{n_1} \oplus zI_{n_2}, \ {\bf P}_-(z) = zI_{n_1} \oplus rI_{n_2} .
$$
 We refer the reader to \cite{jury2023positivity,mccullough2023geometric} for some recent developments on related theory for the annulus; additional references concerning spectral constants will appear in Section~\ref{sec:annulus}.

Our main result, Theorem \ref{mainT},  states the existence of a realization \eqref{real3} with a contractive matrix  ${\tiny \begin{bmatrix} A & B \\ C & D \end{bmatrix}}$, for any rational matrix function $F$ defined on $\Omega$ when a corresponding Agler norm is strictly less than one. In order to prove this result, we will be making use of the realization result in \cite[Theorem 3.4]{ grinshpan2016matrixvaluedhermitianpositivstellensatzlurking}, as well as the operator algebra techniques developed by Blecher, Ruan and Sinclair \cite{Blecher1990}. 
As part of the argument, we show in Theorem~\ref{preptheoremD} that the Agler norm on $\Omega$ equals
$$
\inf\{\norm{G}_{\mathcal T^\circ_{\mathbb D^d}}: G\circ\gamma=F\},   
$$
 where $\gamma$ is the natural embedding of $\Omega$ into the polydisk. 
The equality, arising from the works  \cite{Mittal,MittalPaulsen}, can be also viewed as a result on complete extension sets with respect to the Agler norms,  cf.\ \cite{agler2023complete}. Furthermore, it   has also consequences for the estimations of the Bohr radius for the annulus,  which is the second main outcome of this paper.

Recall that in \cite{Bohr14} Bohr showed his inequality, which can be reformulated  
that for any contractive operator on a Banach space the disc of radius 3 centered at zero is its spectral set, cf.\ \cite{KM66}. Recent activity on the topic can be found in \cite{BHSV2025,knese2024three,paulsen2022}.  In  the present paper we show that for 
a contractive element of  a Banach algebra $T$ satisfying $\norm{T^{-1}}\leq r^{-1}$   ($r<1$)  the annulus  with radii $K_2^{-1}$ and $rK_2$ is a $(1+\sqrt 2)$-spectral set; see Theorem~\ref{Banach}. The constant $K_2$ is the 2-variate version of the Bohr constant, and is known to lie in the interval $(0.3006,1/3)$. We are able to narrow the interval to  $(0.3006,0.3177)$ in Theorem~\ref{K2est}.

 Our paper is organized as follows. In Section~\ref{sec:aglernorm} we provide preparatory  results on the Agler norm on the set of rational functions without poles in a set $\Omega$, which is a bounded intersection of discs in the Riemann sphere. In Section~\ref{s3} we discuss the finite dimensional contractive realization theorem for such domains. In Section~\ref{sec:domains} we provide a sufficient condition for the equality of two Agler norms defined over two sets of operators, distinguished from each other by the boundary behaviour.  These results are used in unifying the various definitions of the Agler norm. In Section~\ref{sec:annulus} we present  special cases, focusing on the annulus. Section~\ref{sec:bohr} is devoted to new results related to the Bohr inequality  for the annulus and the bidisk.

We next discuss notation. The set of bounded operators between Hilbert spaces $\mathcal H_1$ and $\mathcal H_2$ will 
 be denoted by $\mathcal B(\mathcal H_1,\mathcal H_2)$. The set of matrices of size $n\times n$ will be denoted by $\mathcal M_{n}(\Comp)$. In the manuscript we will use several norms on several spaces. In two instances we will not indicate the space in the subscript, namely for $A\in\mathcal B(\mathcal H):=\mathcal B(\mathcal H,\mathcal H)$  the symbol $\norm A$ will stand  its operator norm and for  a matrix $A\in\mathcal M_n(\Comp)$ the norm $\| A\|$ stands for the spectral norm (= induced 2-norm). For a closed set with nonempty interior $\Xi\subseteq\Comp^l$ ($l\geq0$)   by $\mathcal{R}(\Xi)$ we denote the algebra of rational functions, analytic on (some neighbourhood of) $\Xi$ normed with the supremum norm. Subsequently, by $\cM_n(\mathcal{R}(\Xi))  $ 
we understand the algebra of matrix-valued rational functions, analytic on (some neighbourhood of) $\Xi$, endowed
with the  norm 
\begin{equation}\label{RMnorm}
\|F\|_{  \cM_n(\mathcal{R}(\Xi)) }=\sup_{z\in\Xi} \| F(z)\|.
\end{equation}

\section{ The Agler norm on intersection of discs on the Riemann sphere} \label{sec:aglernorm}

In what follows we will consider a bounded intersection $\Omega$  of closed discs on the Riemann sphere. 
More precisely,
\begin{equation}\label{Omegadef}
\Omega=\bigcap_{j=1}^k\Omega_j, \quad \Omega_j= \begin{cases} 
\{ z\in\Comp : |z-\alpha_j|\leq r_j \} &:\ j=1,\ldots, k_1\\ 
\{ z\in\bar\Comp : |z-\alpha_j|\geq r_j \} &:\ j=k_1+1,\ldots, k_2\\
\{ z\in\bar\Comp : \RE(\e^{\ii\theta_j} z)  \leq r_j \}  &:\ j=k_2+1,\ldots, k
\end{cases}.
\end{equation}
where 
$$
1\leq k_1\leq k_2 \leq k,\quad
\begin{array}{c} r_1,\dots,r_{k_2}>0,\\ r_{k_2+1},\dots , r_{k}\geq 0,  \end{array}\ 
\begin{array}{c} \\
0\leq \theta_{k_2+1},\dots , \theta_{k} <2\pi . \end{array}
$$
Note that, by definition, $\Omega_1$ (and hence $\Omega$) is a subset  of $\Comp$.
 Further, we define the related $k\times k$ matrix polynomials
 \begin{eqnarray}\label{PpD}
&\Pp(z)=\diag(r_j)_{j=1}^{k_1}\oplus \diag(z-\alpha_{j})_{j=k_1+1}^{k_2} \oplus\diag( \e^{\ii\theta_j} z -r_j-1 )_{j=k_2+1}^k \\
&\Pm(z)=\diag(z-\alpha_{j})_{j=1}^{k_1}\oplus \diag(r_{j})_{j=k_1+1}^{k_2} \oplus\diag( \e^{\ii\theta_j} z -r_j+1 )_{j=k_2+1}^k \label{PmD}
\end{eqnarray}
Observe that $\mathbf{P}_+(z)$ is invertible for $z \in\Omega$ and
\begin{equation}\label{PgammaOmega}
     \mathbf{P}_-(z)\mathbf{P}_+(z)^{-1}= \diag(\gamma_1(z),\dots,\gamma_N(z))
    \end{equation}
    where $\gamma_j$
   is a M\"obius transform that maps $\Omega_j$ onto the $\overline{\mathbb D}$. Hence,
  $$
\gamma:=(\gamma_1,\dots,\gamma_k):
 \Omega\to \overline{\mathbb D^k}.
  $$

We define the class $\mathcal T_\Omega$ as all those operators $T$ acting on some Hilbert space $\mathcal H$ such that 
\begin{eqnarray}\label{TAr1}
 && T-\alpha_j I \textrm{ is invertible,}\  j={k_1+1},\dots,k_2,\\
 && 
 \e^{\ii\theta_j}T-(r_j+1)I \textrm{ is invertible,}\  j={k_2+1},\dots,k \label{TAr1b},\\
 && \|\gamma_j(T) \| \leq 1,\quad j=1,\dots,k .\label{TAr2}
\end{eqnarray}
\begin{remark}
Note that \eqref{TAr1} and \eqref{TAr1b} are equivalent to $\Pp(T)$ being invertible.
Under this assumption the condition \eqref{TAr2} can be rewritten  equivalently as 
\begin{equation}\label{PTineq}
\Pp(T)^*\Pp(T) - \Pm(T)^*\Pm(T) \geq0.
\end{equation}

\end{remark}
Let us also note that $\gamma(T)=(\gamma_1(T),\ldots,\gamma_k(T))$ is a well-defined tuple of commuting contractions,
for any operator $T\in\cT_{\Omega}$.

Besides the class $\mathcal T_\Omega$ we also define the class $\mathcal T_{\mathbb D^k}$ as the class of all $k$-tuples of commuting Hilbert space (not necessarily strict) contractions.
 For $F\in \mathcal{M}_n(\mathcal R(\Xi))$  ($\Xi=\Omega$ or $\Xi=\overline{\mathbb D}^k$)  we introduce {\em the Agler norm (with respect to the class $\mathcal T_{\Xi}$)}
\begin{equation}\label{Aglerdef}
\|F\|_{\mathcal T_{\Xi}} := \sup\limits_{T \in\mathcal T_{\Xi}} \|F(T)\|.
\end{equation}               
Given an operator $T$ we say that  $\Omega$ is a \emph{$\Psi_{cb}$-complete spectral set for $T$} (where $\Psi_{cb}>0$), if the spectrum of $T$ is contained in $\Omega$ and $$
\| F(T)\|\leq\Psi_{cb} \sup_{z\in\Omega}\| F(z)\|,\quad F\in \mathcal{M}_n(\mathcal  R(\Omega) ).
$$ 
The constant $\Psi_{cb}$ estimates the Agler norm, as the following proposition shows. The result is an adaptation of the result of Badea,  Beckermann and Crouzeix on the absolute bound \cite{BBC2009}.

\begin{proposition}\label{basicAglerD}
Let $\Omega$  be the bounded intersection of discs on the Riemann sphere given by \eqref{Omegadef}. For each $T$ in the class $\mathcal T_\Omega$  the set $\Omega$ is a $\Psi_{cb}$-complete spectral set with $ \Psi_{cb}\leq k+k(k-1)/\sqrt 3$. 
In particular, the Agler norm is finite and bounded from above as
$$
\norm F_{\mathcal T_\Omega}\leq( k+k(k-1)/\sqrt 3)\sup_{z\in\Omega}\| F(z)\|,\quad F\in \mathcal{M}_n(\mathcal  R(\Omega) ).
$$
\end{proposition}

\begin{proof}
Take arbitrary $T$ from $\mathcal T_\Omega$. First observe that due to the definitions of $\Omega$ and $\mathcal T_\Omega$, the spectrum of $T$ is contained in $\Omega$.
Further, each of the discs $\Omega_j$ is a spectral set for $T$. 
Hence, we may apply Theorem 1.1 of \cite{BBC2009} and  obtain that  $\Omega $ is a $\Psi_{cb}$-complete spectral set for $T$ with  $\Psi_{cb}\leq    k+k(k-1)/\sqrt 3$. As $T$ was arbitrary the second claim follows.
\end{proof}

 Let us discuss  now the topic of operator algebras. Let $\mathcal{A}$ be a  normed algebra. Using the identification $\mathcal{M}_{n}(\mathcal{A}) \cong \mathcal{A} \otimes \mathcal{M}_{n}(\mathbb{C})$ we can endow $\mathcal{M}_{n}(\mathcal{A})$ with the natural mutiplication. Suppose a collection of norms $\{\|\cdot\|_{\mathcal{M}_{n}(\mathcal{A})}\}_{n=1}^\infty$ is given.
The following conditions are known as the \textit{Blecher-Ruan-Sinclair} axioms:
\begin{enumerate}[\rm (i)]
    \item for all $n \in \mathbb{Z}_+$ and for all $M, N \in \mathcal{M}_{n}(\mathbb{C})$, $X \in \mathcal{M}_{n}(\mathcal{A})$ we have
    $$
    \|MXN\|_{\mathcal{M}_{n}(\mathcal{A})} \leq \|M\| \|X\|_{\mathcal{M}_{n}(\mathcal{A})} \|N\|.
    $$
    \item If $X \in \mathcal{M}_{n}(\mathcal{A})$ and $Y \in \mathcal{M}_{m}(\mathcal{A})$, then
    $$
    \|X \oplus Y\|_{\mathcal{M}_{m+n}(\mathcal{A})} = \max (\|X\|_{\mathcal{M}_{n}(\mathcal{A})}, \|Y\|_{\mathcal{M}_{m}(\mathcal{A})}).
    $$
    \item If $X, Y \in \mathcal{M}_n(\mathcal{A})$, then 
    $$
    \|XY\|_{\mathcal{M}_{n}(\mathcal{A})} \leq \|X\|_{\mathcal{M}_{n}(\mathcal{A})}\|Y\|_{\mathcal{M}_{n}(\mathcal{A})}.
    $$
\end{enumerate}

Subsequently,  if $\mathcal{H}$ is a Hilbert space, we call $\pi: \mathcal{A} \to \mathcal{B}(\mathcal{H})$ a \textit{unital completely isometric algebra homomorphism} if $\pi$ is an algebra homomorphism, $\pi(e) = I_{\mathcal{H}}$ and 
$$
\|X\|_{\mathcal{M}_{n}(\mathcal{A})} = \| [\pi(X_{ij})]_{ij=1}^n  \|_{\mathcal B(\mathcal H^n)} ,  \quad X = [X_{ij}]_{ij=1}^n \in \mathcal{M}_{n}(\mathcal{A}).
$$

Theorem 3.1 and Corollary 3.2 of \cite{Blecher1990} can be now summarised as follows. 

\begin{theorem} \label{blech}
Assume  $\mathcal A$ is a normed algebra $\mathcal{A}$ with unit of norm $1$, the system of norms
$\{\|\cdot\|_{\mathcal{M}_{n}(\mathcal{A})}\}_{n=1}^\infty$ satisfies the Blecher-Ruan-Sinclair axioms,
and $\mathcal I$ is a two-sided ideal in $\mathcal{A}$. 
Then there exists  a unital completely isometric algebra homomorphism
$\pi: \mathcal{A}/\mathcal I \to \mathcal{B}(\mathcal{H})$, for some Hilbert space $\mathcal{B}(\mathcal{H})$.
\end{theorem}

Clearly, $\mathcal R(\Xi)$ with the matrix norm structure \eqref{RMnorm} is an example of an algebra satisfying (i)--(iii).
For our purposes, however, we  need to norm $\mathcal R(\Xi)$ with the Agler norm.
We present the following lemma for completeness.

\begin{lemma}\label{TRBS}
    The Agler norm \eqref{Aglerdef} is a norm  on $\mathcal{M}_n(\mathcal R(\overline{\mathbb D^k}))$ and it satisfies the Blecher-Ruan-Sinclair axioms.
\end{lemma}

\begin{proof}

To show that the Agler norm is indeed a norm, we need to show that it is finite on
$\mathcal{M}_n(\mathcal R(\overline{\mathbb D^k}))$; the other norm axioms are easily checked. 
Let $F\in \mathcal{M}_n(\mathcal R(\mathbb D^k))$; by definition  $F\in \mathcal{M}_n(\mathcal R((1+\varepsilon)\overline{\mathbb D^k})$ for some $\varepsilon>0$. Taking an arbitrary tuple of commuting contractions $\mathbf T=(T_1,\dots,T_k)$, we have
\begin{eqnarray*}
\norm{F(\mathbf T)}_{\mathcal{M}_n(\mathcal B(\mathcal H))}&   \leq& \frac 1{(2\pi)^k}\int_{|z_1|=1+\varepsilon } \cdots\int_{|z_k|=1+\varepsilon } \sup_{z\in (1+\varepsilon)\mathbb D}\|F(z)\| \\
&&\cdot \norm{(z_1I-T_1)^{-1}}\cdots\norm{(z_kI-T_k)^{-1}} |dz_1|\cdots|dz_k| .
\end{eqnarray*}
To bound the latter independently of $\mathbf T$ note that 
$$\norm{(z_j I - T_j)^{-1}}\leq \frac{1}{|z_j|} \sum_{l=0}^\infty \left( \frac{\| T_j\|}{|z_j|} \right)^l \le \frac{1}{1+\varepsilon} \sum_{l=0}^\infty (1+\varepsilon)^{-l}= \frac 1\varepsilon, \ j=1,\dots,k,$$ producing in total 
$$
\norm{F(\mathbf T)}_{\mathcal{M}_n(\mathcal B(\mathcal H))}\leq\frac{(1+\varepsilon)^k}{\varepsilon ^k} \sup_{z\in (1+\varepsilon)\mathbb D}\|F(z)\| .
$$

Now let us show that the Agler norm satisfies Blecher-Ruan-Sinclair axioms; only  (i) requires some attention as (ii) and (iii) are elementary. 
Let  $F \in \cM_n(\mathcal{R}(\overline{\mathbb{D}^k}))$, $F=\sum_{\alpha\in\mathbb N^k}z^\alpha A_\alpha$,
$M, N \in \cM_n(\mathbb{C})$,  and $\mathbf T=(T_1,\dots,T_k)$ be an arbitrary tuple of commuting contractions
on some Hilbert space $\mathcal H$. We have
   \begin{eqnarray*}
   \|(MFN)(\mathbf T)\|_{\cM_n(\mathcal B(\mathcal H))} &=&\norm{ \sum_{\alpha}T^\alpha\otimes ( MA_\alpha N) }_{\cM_n(\mathcal B(\mathcal H))}\\
   &\leq&
   \|I_{\mathcal H}\otimes  M\|_{\cM_n(\mathcal B(\mathcal H))}   \| F(\mathbf T) \|_{\cM_n(\mathcal B(\mathcal H))} \|I_{\mathcal H}\otimes N\|_{\cM_n(\mathcal B(\mathcal H))} \\
   &\leq&  \|M\| \|F\|_{\mathcal T(\mathbb{D}^k)} \|N\|.
    \end{eqnarray*}
    Passing to supremum with respect to $\mathbf T$ and $\mathcal H$ on the left hand side proves the claim. 

\end{proof}

Observe that $\gamma$ induces an algebra homomorphism
 $\gamma^*: \cM_n(\mathcal{R}(\overline{\mathbb{D}^k})) \to \cM_n(\mathcal{R}(\Omega))$ by the formula 
 $$
 \gamma^*(F)(z) = F(\gamma(z)).
$$

\begin{lemma}\label{surjectivegamma*D}
The linear mapping $\gamma^*: \mathcal{R}(\overline{\mathbb{D}^k}) \to\mathcal{R}(\Omega)$ is a surjection,
contractive with respect to the Agler norms on both sides.
\end{lemma}

\begin{proof}
Given $F \in \mathcal{R}(\Omega)$ we can decompose
it into $F=\sum_{j=1}^k F_j$ with $F_j\in \mathcal{R}(\Omega_j)$, by grouping
fractions in the partial
fraction decomposition of $F$ according to location of their poles.
For $j=1,\ldots,k$ note that $\gamma_j^{-1}:\overline{\mathbb D}\to \Omega_j$ is a rational function.
Then 
\begin{equation}\label{FtoG}
G(z_1,\ldots,z_k):=\sum_{j=1}^k F_j(\gamma_j^{-1}(z_j))
\end{equation}
 is a well defined
rational function on $\overline{\mathbb D}^k$ and $\gamma^*(G)=F$.
   
To prove contractivity observe that for $G\in \cM_n(\mathcal R(\overline{\mathbb D^k})))$ we have 
$$
\norm{ \gamma^*(G)( T)  } =\norm{G(\gamma(T))}\leq \norm{ G   }_{\mathcal T(\mathbb D^k)}
$$
for any $T$ in $\mathcal T(\Omega)$, where the inequality follows from $\norm{\gamma_j(T)}\leq 1$, $j=1,\dots,k$.
Passing to the supremum over $T$ finishes the proof.
\end{proof}

We end up this preparatory section with a key result, which will be used subsequently both for
contractive realization and Bohr radius results. In order to do this, we define the quotient norm
$\|\cdot\|_{\mathcal T\ker\gamma^*}$ on $\cM_n(\mathcal{R}(\Omega))$ by
$$
\|F\|_{\mathcal T\ker\gamma^*} := \inf \{\|G\|_{\mathcal{T}( \mathbb{D}^k)}: G \in \cM_n(\mathcal{R}(\overline{\mathbb{D}^k})), \gamma^*(G) = F\}.
$$
The definition is correct due to  Lemma~\ref{surjectivegamma*D}.

\begin{theorem} \label{preptheoremD}
Let $\Omega$  be the bounded intersection of discs on the Riemann sphere given by \eqref{Omegadef}.
Then the quotient norm defined above and the Agler norm coincide, i.e., 
$$
\|F\|_{\mathcal T\ker\gamma^*} =  \|F\|_{\mathcal T_\Omega},\quad F \in \cM_m(\mathcal{R}(\Omega)).
$$
\end{theorem}

\begin{proof}
First we show that $ \|F\|_{\mathcal T\ker\gamma^*} \leq \|F\|_{\mathcal{T}_{\Omega }}$.
Due to Lemma~\ref{TRBS}we have that $\cM_n(\mathcal{R}(\overline{\mathbb{D}^k}))$ with the Agler norm structure
satisfies the Blecher-Ruan-Sinclair axioms.
By  Theorem~\ref{blech} there exists a completely isometric homomorphism
$\phi: \cM_n(\mathcal{R}(\Omega)) \to \cB(\mathcal{H})$, where $\cM_n(\mathcal{R}(\Omega))$ is equipped with
the system of quotient Agler norms $\|\cdot\|_{\mathcal T\ker\gamma^*}$.

Let $T: = \phi(z)$, we claim that $T$ belongs to $\mathcal T_\Omega$. Indeed, 
$$
(T-\alpha_jI_{\mathcal H})\phi( 1/(z-\alpha_j))=\phi(z-\alpha_j)\phi( 1/(z-\alpha_j))=\phi( 1)=I_{\mathcal H},
$$
for $j=k_1+1,\dots,k_2$, thus \eqref{TAr1} holds.
 Analogously, for $j=k_2+1,\dots,k$, 
 $$
\big(\e^{\ii\theta_j}T-(r_j-1)I_{\mathcal H}\big)\phi\big( 1/(\e^{\ii\theta_j}z-r_j+1)\big)=\phi\Big(\frac{\e^{\ii\theta_j}z-r_j+1}{\e^{\ii\theta_j}z-r_j+1}\Big)=
I_{\mathcal H},
$$
proving \eqref{TAr1b}.
To see \eqref{TAr2} observe that as $\phi$ is isometric we have
$$
\| \gamma_j(T)\|_{\mathcal{B}(\mathcal H)}= \| \phi (\gamma_j) \|_{\mathcal{B}(\mathcal H)} =\| \gamma_j \|_{\mathcal T\ker\gamma^*}\leq \| z_j\|_{\mathcal T_{\mathbb D^k}} = 1,
$$
where the  inequality follows from the fact that the function $Q(z)=z_j$ satisfies $Q(\gamma(z))=\gamma_j(z)$.

Now for any $F\in \cM_m(\mathcal{R}(\Omega) )$ we have 
$$
\|F\|_{\mathcal T\ker\gamma^*} = \|\phi(F)\|_{\mathcal{B}(\mathcal H)} = \|F(T)\|_{\mathcal{B}(\mathcal H)} \leq
\|F\|_{\mathcal T_\Omega}. 
$$
  
To show the inequality $ \|F\|_{\mathcal T\ker\gamma^*} \geq \|F\|_{\mathcal{T}_\Omega}$ take
$\varepsilon>0$ and $G\in \cM_m(\mathcal  R(\overline{\mathbb{D}^k }))$, such that 
\begin{equation}\label{FG}
\gamma^*(G) = F,\quad \|G\|_{\mathcal T_{\mathbb{D}^k}} < \|F\|_{\mathcal{T} \ker\gamma^*}+\varepsilon.
\end{equation}
Hence, for any $T\in\mathcal T_\Omega$ we have  $\gamma(T)\in\mathcal T_{\mathbb D^k}$ and 
$$
\|F(T)\|_{\mathcal B(\mathcal H)} = \|G(\gamma(T))\| \leq
	\|G\|_{\mathcal T_{\mathbb{D}^k}} <\|F\|_{\mathcal{T}\ker \gamma^*} +\varepsilon .
$$
Taking the supremum over $T \in \Omega$ and noting that $\varepsilon>0$ was arbitrary we obtain the claim. 
\end{proof}

\section{  Contractive realization on intersection of discs on the Riemann sphere} 
\label{s3}

We are ready to present our  results on contractive realization. The section contains three of them. 
First, we provide the central result regarding the contractive realization \eqref{real3}. Second, we show that any function having such realization is contractive in the Agler norm. We derive this argument for a very general, possibly infinite dimensional, class of realizations. 
Third, we employ the Bremehr's result, to replace the Agler norm in the assumption by the supremum norm.
Below we abbreviate $\mathbf P_\pm(z)\otimes I_m$ by $\mathbf P_\pm(z)_m $.

\begin{theorem} \label{mainT}
 Let $\Omega$  be the bounded intersection of discs on the Riemann sphere given by \eqref{Omegadef}. Then  for any  rational matrix-valued function $F\in \cM_n( \mathcal{R}(\Omega) )$ with   the Agler norm satisfying $\|F\|_{\mathcal T_\Omega}<1 $, there exists a positive integer $m$ and a 
contractive matrix
   {\tiny  $\begin{bmatrix} A & B \\ C & D \end{bmatrix}$} of size $(km+n) \times (km+n)$
    such that
    \begin{equation}\label{F-ABCD} F(z) = D +C\mathbf{P}_-(z)_{m}(\mathbf{P}_+(z)_{m}-A\mathbf{P}_-(z)_{m})^{-1}B, 
    \end{equation}
    where $\mathbf P_\pm(z)$ are defined by \eqref{PpD} and \eqref{PmD}.
    
   Furthermore, if $F\in \cM_n( \mathcal{R}(\Omega) )$
    has a realization \eqref{F-ABCD} with a contractive matrix
   {\tiny  $\begin{bmatrix} A & B \\ C & D \end{bmatrix}$} then 
   $$
   \|F\|_{\mathcal T_\Omega}\leq \norm{\begin{bmatrix} A & B \\ C & D \end{bmatrix} } .
   $$
\end{theorem}

\begin{proof}

By Theorem~\ref{preptheoremD} and the definition of the norm $\|F\|_{\mathcal T\ker\gamma^*}$ there exists $G \in \cM_n(\mathcal{R}(\overline{\mathbb D^k}))$ such that $G(\gamma(z)) = F(z)$ and $\| G\|_{\mathcal T_{\mathbb{D}^k}} < 1$. 
By Theorem~3.4 from \cite{grinshpan2016matrixvaluedhermitianpositivstellensatzlurking} there exist $m_1,\dots, m_k\in \mathbb{Z}_+$ and a contractive colligation matrix {\tiny $\begin{bmatrix} A & B \\ C & D \end{bmatrix}$} of size $(\sum m_j+n) \times (\sum m_j+n)$ such that  
$$
G(z) = D + CZ(I - AZ)^{-1}B,
$$
where $Z = z_1I_{m_1} \oplus\cdots\oplus z_kI_{m_k}$. 
Appending rows and columns to the colligation matrix we may assume that $m_1=\cdots=m_k=:m$, hence $Z=\diag(z_1,\dots,z_k)\otimes I_m$.
Employing \eqref{PgammaOmega} we receive
\begin{eqnarray*} 
    F(z) &= &  G(\gamma(z)) \\
    &=&    D + C(\diag(\gamma(z))\otimes I_m))(I - A(\diag(\gamma(z))\otimes I_m ))^{-1}B\\
    &=& D + {C} \mathbf{P}_-(z)_{m}\mathbf{P}_+(z)_{m}^{-1}\left(I - {A} \mathbf{P}_-(z)_{m}\mathbf{P}_+(z)_{m}^{-1}\right)^{-1} {B} \\ 
    &=& D +C\mathbf{P}_-(z)_{m}(\mathbf{P}_+(z)_{m}-{A}\mathbf{P}_-(z)_{m})^{-1}{B}
    \end{eqnarray*}
    for all $z \in \Omega$, as desired.
    
    The second part of the statement follows from a more general fact, Proposition~\ref{Jconminreal} below.
\end{proof}

We next show that every function $F(z)$ with realization \eqref{real3}  (finite or infinite dimensional; one or many variables)
has Agler norm less or equal to one. The type of calculation the proof requires has appeared in special cases before.
For instance, when ${\tiny \begin{bmatrix} A & B \\ C & D \end{bmatrix}}$ is an isometry and ${\bf P}_+\equiv I$,
the calculation appears in \cite[Section 6.2]{AglerMcCarthy}. A more general case appears in \cite{Jackson}.

\begin{proposition} 
\label{Jconminreal} Suppose that ${{\Xi} }\subseteq\Comp^m$ is any domain   and that operator-valued polynomials
$\mathbf{P}_+(z)=\sum_{\alpha} {\bf P}_{+,\alpha}z^\alpha:{\mathcal H} \to {\mathcal H}$ and $\mathbf{P}_-(z)=\sum_{\alpha} {\bf P}_{-,\alpha}z^\alpha:{\mathcal H}\to {\mathcal K}$, $z\in \Xi$, 
 satisfy \begin{equation}
      \mathbf{P}_+(z)^*\mathbf{P}_+(z)-\mathbf{P}_-(z)^*\mathbf{P}_-(z)>0 , \quad z\in\Xi.\label{P*Pz}
\end{equation}
Then for any analytic ${\mathcal B}({\mathcal U}, {\mathcal Y})$-valued function $F(z)$ with realization \eqref{real3} with $A,B,C,D$
satisfying \begin{equation}\label{3} \begin{bmatrix} A & B \\ C & D \end{bmatrix} :\begin{matrix} {\mathcal K} \cr \oplus \cr {\mathcal U} \end{matrix} \to \begin{matrix} {\mathcal H} \cr \oplus \cr {\mathcal Y} \end{matrix} , \ \ \ \ \ \ 
 \left\| \begin{bmatrix} A & B \\ C & D \end{bmatrix} \right\| \le 1,
 \end{equation} and for any  tuple of commuting Hilbert space operators ${\bf T}$  with ${\bf P}_+({\bf T})$ invertible and 
 satisfying
 \begin{equation}\label{3a} \mathbf{P}_+({\bf T})^*\mathbf{P}_+({\bf T})-\mathbf{P}_-({\bf T})^*\mathbf{P}_-({\bf T})>0 ,\ \ \ \ \mathbf{P}_\pm({\bf T})=\sum_{\alpha} {\bf P}_{\pm,\alpha}\otimes T^\alpha ,\end{equation} the operator $F({\bf T})$ is well defined and 
 \begin{equation}\label{F<}
     \| F({\bf T}) \|\leq \left\| \begin{bmatrix} A & B \\ C & D \end{bmatrix} \right\| .
 \end{equation}
\end{proposition}

\begin{proof} Note that if ${\bf T}\in {\mathcal B}({\mathcal L})^m$, we have that 
$$
F({\bf T}) = D_e +C_e{\bf P}_-({\bf T}) ({\bf P}_+({\bf T}) -A_e{\bf P}_- ({\bf T}))^{-1} B_e \in {\mathcal B}({\mathcal U}\otimes {\mathcal L}, {\mathcal Y}\otimes{\mathcal L}), 
$$
where
$$
A_e=A \otimes I_{{\mathcal L}},\ B_e=B\otimes I_{\mathcal L},\ C_e=C\otimes I_{\mathcal L} ,\ D_e=D\otimes I_{\mathcal L}.
$$
For convenience, we will use the notation $Z_\pm=\mathbf{P}_\pm({\bf T})$. Let us check the well-definedness of $F({\bf T})$.
Since $A_e$ is a contraction, we have that 
$$  Z_+^* Z_+ >Z_-^* Z_-  \ge Z_-^* A_e^*A_eZ_- , $$
which implies that 
$$ I > Z_+ ^{*-1}  Z_-^* A_e^*A_eZ_-  Z_+ ^{-1}. $$
Thus $ \|  A_e Z_- Z_+ ^{-1}  \| <1$, giving that 
$$ Z_+ -A_eZ_-  = (I-A_eZ_- Z_+ ^{-1}) Z_+ $$
is invertible. Thus $F({\bf T})$ is well defined.

In order to prove \eqref{F<}, we write due to \eqref{3}, 
\[
\begin{bmatrix} P&Q\\Q^*&R\end{bmatrix} := \left( \begin{bmatrix} I & \\
& I\end{bmatrix} - \begin{bmatrix}A^*&C^*\\B^*&D^*\end{bmatrix}\begin{bmatrix}A&B\\C&D\end{bmatrix} \right)
\otimes I_{\mathcal L} \ge \varepsilon I \geq 0,
\]
where $\varepsilon:=1-{\tiny \norm {  \begin{bmatrix} A & B \\ C & D \end{bmatrix} }}^2\geq0 $,
obtaining the identities
\begin{eqnarray*}
 P &=& I- A_e^*A_e - C_e^*C_e,\\
 Q &=& -A_e^*B_e - C_e^*D_e,\\
 R &=& I - B_e^*B_e - D_e^*D_e.
\end{eqnarray*}
Using the above observations, we now have
\begin{align*}
I - F(\mathbf T)^*F(\mathbf T) =& I - D_e^*D_e - B_e^*(Z_+^*-Z_-^*A_e^*)^{-1}Z_-^*C_e^*D_e \\ 
&- D_e^*C_eZ_-(Z_+-A_eZ_-)^{-1}B_e \\
&- B_e^*(Z_+^*-Z_-^*A_e^*)^{-1}Z_-^*C_e^*C_eZ_-(Z_+-A_eZ_-)^{-1}B_e
\end{align*}
\begin{align*}
=& R+B_e^*B_e+ B_e^*(Z_+^*-Z_-^*A_e^*)^{-1}Z_-^*(Q+A_e^*B_e) \\
&+ (Q^*+B_e^*A_e)Z_-(Z_+-A_eZ_-)^{-1}B_e\\ 
&+ B_e^*(Z_+^*-Z_-^*A_e^*)^{-1}Z_-^*(P-I+A_e^*A_e)Z_-(Z_+-A_eZ_-)^{-1}B_e
\end{align*}
\begin{align*}
=& R+B_e^*(Z_+^*-Z_-^*A_e^*)^{-1}Z_-^*Q + Q^*Z_-(Z_+-A_eZ_-)^{-1}B_e  \\
&+ B_e^*(Z_+^*-Z_-^*A_e^*)^{-1}Z_-^*PZ_-(Z_+-A_eZ_-)^{-1}B_e \\ 
&+ B_e^*B_e+ B_e^*(Z_+^*-Z_-^*A_e^*)^{-1}Z_-^*A_e^*B_e + B_e^*A_eZ_-(Z_+-A_eZ_-)^{-1}B_e \\
&+ B_e^*(Z_+^*-Z_-^*A_e^*)^{-1}Z_-^*(A_e^*A_e - I)Z_-(Z_+-A_eZ_-)^{-1}B_e,
\\
 =& \begin{bmatrix} B_e^*(Z_+^*-Z_-^*A_e^*)^{-1}Z_-^*
 & I\end{bmatrix}\begin{bmatrix} P&Q\\Q^*&R\end{bmatrix}\begin{bmatrix}Z_-(Z_+-A_eZ_-)^{-1}B_e \\ I\end{bmatrix} \\
&  +B_e^*(Z_+^*-Z_-^*A_e^*)^{-1}\Big[ (Z_+^*-Z_-^*A_e^*)(Z_+-A_eZ_-)   + Z_-^*A_e^*(Z_+-A_eZ_-) \\
&+ (Z_+^*-Z_-^*A_e^*)A_e Z_- +Z_-^* A_e^* A_e Z_- - Z_-^*Z_- \Big](Z_+-A_eZ_-)^{-1}B_e
\\
 =& \begin{bmatrix} B_e^*(Z_+^*-Z_-^*A_e^*)^{-1}Z_-^*
 & I\end{bmatrix}\begin{bmatrix} P&Q\\Q^*&R\end{bmatrix}\begin{bmatrix}Z_-(Z_+-A_eZ_-)^{-1}B_e \\ I\end{bmatrix} \\
 & +  B_e^*(Z_+^*-Z_-^*A_e^*)^{-1}\Big[ Z_+^*Z_+-Z_-^*Z_-\Big](Z_+-A_eZ_-)^{-1}B_e.
\end{align*}
As $Z_+^*Z_+ - Z_-^* Z_-$ is positive definite, we recieve
\begin{eqnarray*}
\langle (I & - & F(\mathbf T)^*F(\mathbf T))x,x\rangle  \\
&\geq& 
 \left\langle \begin{bmatrix} P&Q\\Q^*&R\end{bmatrix}\begin{bmatrix}Z_-(Z_+-A_eZ_-)^{-1}B_ex \\ x\end{bmatrix},
 \begin{bmatrix}Z_-(Z_+-A_eZ_-)^{-1}B_ex \\ x\end{bmatrix}  \right\rangle\\
 & \geq&  \varepsilon \left\langle \begin{bmatrix}Z_-(Z_+-A_eZ_-)^{-1}B_ex \\ x\end{bmatrix},
 \begin{bmatrix}Z_-(Z_+-A_eZ_-)^{-1}B_ex \\ x\end{bmatrix}  \right\rangle  \\ 
 & \geq& \varepsilon \norm x^2
,\quad x\in\mathcal U.  
\end{eqnarray*}
Hence,
\[
\norm{   F(\mathbf T)x} ^2 \leq (1-\varepsilon) \norm {x}^2
= {\tiny \norm {  \begin{bmatrix} A & B \\ C & D \end{bmatrix} }}^2\norm {x}^2,\quad x\in\mathcal U.
\]
\end{proof}

The estimates of the Agler norm from Proposition \ref{basicAglerD} are clearly not optimal.
Therefore, in Theorem~\ref{thm-brehmer} below we assume that $F$ is such that $\gamma^*(G)=F$ for
some $G$ of supremum norm one. Note this implies that the supremum norm
of $F$ is less or equal to one, but the Agler not necessarily (unless $k=2$). Hence, the
assumptions on $F$ are essentially weaker than in Theorem~\ref{mainT}.
The price for this is a realisation that  takes place on a subset $\check\Omega$ of $\Omega$.
Namely, for $\Omega$, $\mathbf P_{\pm}$, $\gamma$ defined in the usual way by \eqref{Omegadef}--\eqref{PgammaOmega} 
we define 
\begin{equation}\label{checks}
 \check \gamma= (k-1)^{1/2}\gamma,\     \check\Omega:=\check\gamma^{-1}(\overline{\mathbb D^k}\big)\subseteq \Omega,\ \mathbf{\check P}_+:=(k-1)^{-1/2}\Pp, \text{  and }\mathbf{\check P}_-:=\Pm.
\end{equation}
Note that when passing from $\Omega$ to $\check\Omega$, each disc $\Omega_j$ for $j=1,\ldots,k_1$
is scaled by factor $1/\sqrt{k-1}$, and each hole ($\Omega_j$ for $j=k_1+1,\ldots,k_2$) is upscaled by $\sqrt{k-1}$.
The least intuitive is the difference introduced by the half-planes $\Omega_j$ for $j=k_2+1,\ldots,k$,
since the corresponding ``compounds'' of $\check\Omega$ become discs.
In Figure~\ref{fig:omega2} we present an example of $\Omega$ and $\check\Omega$ where
 $k=3$, $k_1=1$, $k_2=2$, $\alpha_1=0$, $\alpha_2=1/2$, $r_1=r_3=1$, $r_2=1/4$ and $\theta_3=0$.
In the second example $\check\Omega$ no longer has a hole, while $\Omega$ has three.
It may actually happen that $\check\Omega=\emptyset$ for a nonempty $\Omega$.
\begin{figure}[ht]
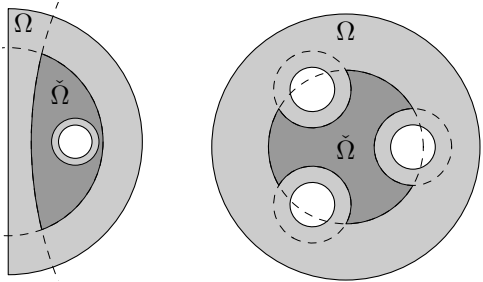

\centering \includegraphics{annulus_03}\qquad \includegraphics{annulus_04}

\caption{Examples of a set $\Omega$ and the corresponding $\check\Omega$ for $k=3,4$.}\label{fig:omega2}
\end{figure}

\begin{theorem}\label{thm-brehmer}
Let $\Omega$  be the bounded intersection of discs on the Riemann sphere given by \eqref{Omegadef}.
Assume $F\in \cM_n( \mathcal{R}(\Omega) )$ is such that there exists $G\in \cM_n(\mathcal R(\overline{ \mathbb{D}^k }))$
with $\gamma^*(G)=F$, and $\sup_{z\in \overline{ \mathbb{D}^k }}  \norm{G(z)}\leq 1$.
Then there exists a positive integer $m$ and a contractive matrix
{\tiny    $\begin{bmatrix} A & B \\ C & D \end{bmatrix}$}
of size $(km+n) \times (km+n)$ such that
\begin{equation}\nonumber
    F(z) = D +C\mathbf{\check P}_-(z)_{m}\bigl(
    \mathbf{\check P}_+(z)_{m}-A\mathbf{\check P}_-(z)_{m}\bigr)^{-1}B, \quad z\in \check\Omega,
    \end{equation}
    where $\check\Omega$ and $\check{\mathbf P}_{\pm}$ are defined by \eqref{checks}.  In particular, 
    $$
    \norm{F}_{\mathcal T_{\check\Omega}}\leq 1.
    $$
\end{theorem}

\begin{proof}
Consider the function 
$$
G_1 (z_1,\ldots,z_k):=G\big( (k-1)^{-1/2}z_1,\ldots,(k-1)^{-1/2}z_k\big)
$$
and observe it satisfies $G_1(\check\gamma(z))=F(z)$ and $\sup_{z\in (k-1)^{1/2}\overline{\mathbb D^k} }\norm{G_1(z)}\leq 1$.
By the celebrated result of Brehmer \cite{Breh61} improved very recently by Knese \cite{knese2024three},
we have $\norm{G_1}_{\cT_{\mathbb D^k}}\leq 1$.
Applying Theorem~3.4 from \cite{grinshpan2016matrixvaluedhermitianpositivstellensatzlurking},
in the same way as in the proof of Theorem~\ref{mainT},
we obtain $m\in \mathbb{Z}_+$ and a contractive colligation matrix 
{\tiny $\begin{bmatrix} A & B \\ C & D \end{bmatrix}$} of size $(mk+n) \times (mk+n)$ such that  
$$
G_1(z) = D + CZ(I - AZ)^{-1}B,
$$
with $Z = \diag(z)\otimes I_m$.
Then, for $z\in \check\Omega$, we have
\begin{align*} 
    F(z) &=   G_1(\check\gamma(z)) \\
    &=  D + C(\diag(\check\gamma(z))\otimes I_m)
    \left[I {-} A\bigl(\diag(\check\gamma(z)\bigr)\otimes I_m )\right]^{-1}\!\!B\\
    &= D + C\mathbf{\check P}_-(z)_{m}\mathbf{\check P}_+(z)_{m}^{-1}
    \left(I - A \mathbf{\check P}_-(z)_{m}\mathbf{\check P}_+(z)_{m}^{-1}\right)^{-1}\!B \\ 
    &
    = D+ C\mathbf{\check P}_-(z)_{m}\left(\mathbf{\check P}_+(z)_{m}-A\mathbf{\check P}_-(z)_{m}\right)^{-1}\!\!B,
\end{align*} 
and the proof of realisation formula is complete. The second statement follows directly from Proposition~\ref{Jconminreal}.
\end{proof}


\section{Agler norms with respect to open domains}
\label{sec:domains}
 While we have so far focused on the closed domain $\Omega$,  in the literature
open domains are commonly considered (e.g.\ \cite{MittalPaulsen}), especially
upon studying algebras of bounded analytic functions $H^\infty(\Omega)$.
In this section we discuss the Agler norm with respect to the open domain.
Note that there exist different definitions of underlying classes of operators (see e.g.\ \cite[Subsection 4.2.2]{Mittal}).

Consider the class $\cT_\Omega^\circ$ of Hilbert space operators satisfying \eqref{TAr1}, \eqref{TAr1b}, and the strong inequality in \eqref{TAr2}, 
equivalently \eqref{TAr1}, \eqref{TAr1b}, and
\begin{equation}\label{TArstrong}
\Pp(T)^*\Pp(T) - \Pm(T)^*\Pm(T) > 0.
\end{equation}
Based on this class we define the corresponding Agler norm
\begin{equation}\label{Aglercirc}
    \norm{F}_{\cT_\Omega^\circ}:=\sup_{A\in \cT_\Omega^\circ}\norm{F(A)},\quad F\in \cM_n(\mathcal R(\Omega)).
\end{equation}
In this section we provide sufficient conditions on the set $\Omega$ under which the norms 
$\norm{\cdot}_{\cT_\Omega}$ and $\norm{\cdot}_{\cT_\Omega^\circ}$ coincide.
\footnote{This issue appeared in the multivariable case in \cite{grinshpan2016matrixvaluedhermitianpositivstellensatzlurking},
where one of the assumptions was that elements in $\cT_\Omega$ can be approximated in norm by elements in $\cT_\Omega^\circ$.
When ${\bf P}_+\equiv I$ and ${\bf P}_-(z)$ is linear in $z$, which were the applications in
\cite{grinshpan2016matrixvaluedhermitianpositivstellensatzlurking}, this norm approximation was easily established
by using $\lim_{r\to 1-} r{{\bf T}}={\bf T}$.}

\begin{proposition}\label{Agler<}
Let $\Omega$  be the bounded  intersection of discs on the Riemann sphere given by \eqref{Omegadef}.
Then $\norm{\cdot}_{\cT_\Omega}= \norm{\cdot}_{\cT_\Omega^\circ}$, provided one of the following conditions holds:
\begin{enumerate}[\rm(i)]
\item $\operatorname{Int}(\Omega)\neq\emptyset$ and $k_1=k_2$ $($i.e.\ $\Omega$ is convex$)$,
\item $k_1=1$, $k_2=k\geq 2$ and
\begin{align}
|\alpha_j|+r_j< r_1 &\qquad (j=2,\ldots,k),
\label{holes-inside}\\
|\alpha_j|^2-r_j^2 = |\alpha_2|^2-r_2^2>0 &\qquad (j=3,\ldots,k).
\label{holes-ring}
\end{align}
\end{enumerate}
\end{proposition}

While (i) covers all possible (non-degenerate) convex sets $\Omega$,
(ii) resolves some special cases of a multi-holed disc (cf.\ Corollary~\ref{cor:decentered}).
The condition \eqref{holes-ring} says that  the half-lines
originating at zero are tangent to the holes at the same distance from zero, i.e.\ $d_0:=\sqrt{|\alpha_j|^2-r_j^2}$
(cf.\ Figure~\ref{fig:multihole}). Cases (ii) and (iii) cover the annuli
with arbitrary size and position of the hole.
\begin{figure}[ht]
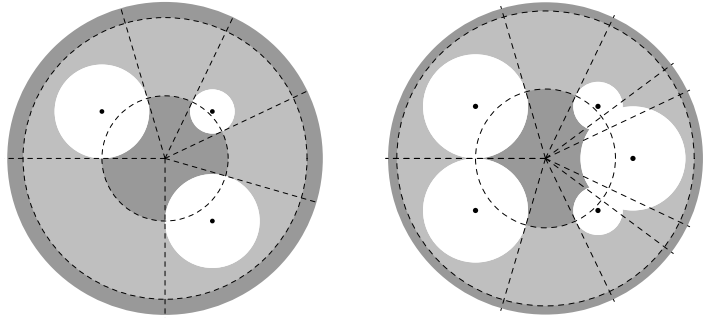

    \centering
    \includegraphics[height=11em]{annulus_01}\qquad
    \includegraphics[height=11em]{annulus_02}
\caption{Examples of ``multi-holed discs''. Note that we do not require the holes to
be disjoint.}

\label{fig:multihole}
\end{figure}

\begin{proof}
To prove the claim it is enough to show that for a fixed Hilbert space $\cH$
the set of $T\in\mathbf B(\cH)$  satisfying \eqref{TAr1}, \eqref{TAr1b} and
\eqref{TArstrong} is dense in the set of $T\in\mathbf B(\cH)$ satisfying \eqref{TAr1}, \eqref{TAr1b}, and  \eqref{TAr2}.
Having this, we infer, e.g.\ using the Cauchy integral formula (cf.~the proof of Lemma~\ref{TRBS}),
that $F(T_\eps)\to F(T)$ in the operator norm, provided that $T_\eps \to T$ as $\varepsilon\to0$.
Therefore $\sup_{S\in\cT^\circ_\Omega}\norm{F(S)} > \norm{F(T)}-\nu$ for arbitrarily small $\nu>0$ and $T\in\cT_\Omega$.
Passing to the supremum we obtain $\norm{F}_{\cT^\circ_\Omega}\geq \norm{F}_{\cT_\Omega}$. The reverse inequality
follows from the inclusion $\cT^\circ_\Omega\subseteq \cT_\Omega$.

(i)
We claim that for $p\in\operatorname{Int}\Omega$, $T\in \cT_\Omega$
and sufficiently small $\eps>0$,
we have $T_\varepsilon:=(1-\eps)T+\eps pI\in\cT_\Omega^\circ$.
It is enough to show this property for a single $\Omega_j$,
since $\cT_\Omega=\bigcap_{j=1}^k\cT_{\Omega_j}$ and 
$\cT_\Omega^\circ=\bigcap_{j=1}^k\cT_{\Omega_j}^\circ$,
where $\cT_{\Omega_j}$ ($\cT_{\Omega_j}^\circ$ resp.) is the class of operators
satisfying \eqref{TAr1b} and \eqref{TAr2} (strong inequality in \eqref{TAr2})
for the particular index $1\leq j\leq k$.

In the case $j\leq k_1$ the set $\Omega_j$ is a disc and we have $\norm{T-\alpha_jI}\leq r_j$ and
$|p-\alpha_j|<r_j$. Then
\begin{align*}
\norm{T_\varepsilon-\alpha_j I}=&\norm{T+\varepsilon(pI-T)-\alpha_j I}\\
=&\norm{(1-\varepsilon)(T-\alpha_j I) + \varepsilon(p-\alpha_j)I}\\
\leq&\, (1-\varepsilon) r_j + \varepsilon|p-\alpha_j| < r_j.
\end{align*}

Now consider the case $j\geq k_2+1$ when the set  $\Omega_j$ is a halfplane. 
Take $T\in {\cT}_{\Omega_j}$, i.e.\ $\e^{\ii\theta_j}T-(r_j-1)I$ is invertible
and 
$ \mathbf{P}^{(j)}_\Delta(T):= \pi_j(\Pp(T)^*\Pp(T)-\Pm(T)^*\Pm(T))\pi_j\geq 0$,
where $\pi_j$ denotes the orthogonal projection of $\Comp^k$ onto the $j$-th coordinate.
Fix $p\in \operatorname{Int}\Omega_j$. Then
\begin{align*}
4\real(\e^{\ii\theta_j}T_\varepsilon) - 2r_j I
&=4\real\big(\e^{\ii\theta_j}(\varepsilon pI+ (1-\varepsilon)T)\big) - 2r_j(\varepsilon+(1-\varepsilon))I
\\&=(1-\varepsilon)\big(4\real( \e^{\ii\theta_j}T) - 2r_j I\big)+\varepsilon\big(4\real( \e^{\ii\theta_j}p) - 2r_j\big)I
\\&\geq 0 + \varepsilon\big(4\real(\e^{\ii\theta_j}p) - 2r_j\big)I > 0.
\end{align*}
Note that $|\gamma_j(p)|<1$ if and only if $4\real(\e^{\ii\theta_j}p) -2r_j >0$.
For sufficiently small $\eps$ the operator $\e^{\ii\theta_j}T_\eps-(r_j-1)I$ is invertible. Employing  $\mathbf{P}^{(j)}_\Delta(T_\eps)=4\real( \e^{\ii\theta_j}T_\varepsilon) - 2r_j I >0$
we deduce that $\norm{\gamma_j(T_\eps)}<1$.

(ii)
Let $R=r_1$, in view of \eqref{holes-ring}  we define
$$
d_0:=\sqrt{|\alpha_2|^2-r_2^2}=\sqrt{|\alpha_j|^2-r_j^2},\quad j=2,\ldots,k.
$$ 
and $d_1:=\frac12(R+\max\{r_j+|\alpha_j|:j=2,\ldots,k\})$.
Then
\[ 0<d_0 < |\alpha_j| < |\alpha_j|+r_j< d_1 < R\quad (j=2,\ldots,k), \]
and for sufficiently small $\eps>0$ we have
\[
(1+\eps)d_1 < R,\quad (1-\eps)d_1 > |\alpha_j|+r_j.
\]

Consider $T\in\cB(\cH)\cap \cT_\Omega$ and let $T=U|T|$ be its polar decomposition ($|T|=(T^*T)^{1/2}$
and $U$ is a partial isometry with $\ker U=\ker T$). Using spectral theorem
for $|T|$ we can find $M_l'\in\cB(\cH_l,\cH_l)$ for $l=0,1,2$ such that
$|T|=M_0'\oplus M_1'\oplus M_2'$, $\sigma(M_0)\subseteq [0,d_0]$,
$\sigma(M_1)\subseteq (d_0,d_1]$ and $\sigma(M_2)\subseteq (d_1,R]$.
Let $M_l:=UM_l'\in \cB(\cH_l, U\cH_l)$ ($l=0,1,2$).
We have $T=M_0\oplus M_1\oplus M_2$. In particular, $\norm{Tx}=\norm{M_l'x}$,
provided that $x\in \cH_l$.

Let $S_{0,\eps}:=(1-\eps)M_0$, $S_{1,\eps}:=(1+\eps)T_1$, $S_{2,\eps}:=(1-\eps)T_2$
and $T_\eps:=S_{0,\eps}\oplus S_{1,\eps}\oplus S_{2,\eps}$. 
Clearly $T_\eps\stackrel{\eps\to 0}{\longrightarrow} T$.
What is left to show is that $T_\eps\in \cT_\Omega^\circ$. Observe that it is enough to show that for $\varepsilon>0$
there exists $\nu>0$ such that 
\begin{equation}\label{frombelow}
\|T_\eps x\|\leq R-\nu,\quad \|T_\eps x-\alpha_j x\|\geq r_j+\nu,\quad \|x\|=1.
\end{equation}
Indeed,  for $\varepsilon>0$ sufficiently small, the point $\alpha_j$ will lie in the resolvent set of $T_\eps$
as $T-\alpha_jI $ is invertible and the set of invertible operators is open. In consequence, for such small $\varepsilon>0$,
condition \eqref{frombelow} implies that  $T_\eps\in \mathcal T^\circ_\Omega$.

Now we prove \eqref{frombelow}. Directly from the definition we have $\|S_{l,\eps}x\|\leq \max\{ (1-\eps)R, (1+\eps)d_1\} < R$.
It remains to show that $\|S_{l,\eps}x-\alpha_j x\|\geq C_{l,j} > r_j$
for $x\in\cH_l$, $\|x\|=1$, $j=2,\ldots,k$, and $l=0,1,2$.
Note that since $\|Tx-\alpha_j x\|\geq r_j$, we have
\begin{equation}\label{eq:reTl}
2\real\scp{M_l x,\alpha_jx}-\|M_l x\|^2 \leq |a_j|^2-r_j^2 = d_0^2
\end{equation}
for $x\in\cH_l$, $\|x\|=1$, $j=2,\ldots,k$ and $l=0,1,2$.

Take $l=0$. Note that for $\|x\|=1$ we have $\|M_0 x\|^2 \leq d_0^2=|\alpha_j|^2-r_j^2$,
which combined with \eqref{eq:reTl} produces 
$\real\scp{M_0x,\alpha_j x}\leq d_0^2$ and
\begin{align*}
\|S_{0,\eps} x - \alpha_j x\|^2 =\,& \|(1-\eps)(M_0 x-\alpha_j x) +\eps\alpha_j x\|^2\\
=\,& (1-\eps)^2\|M_0 x-\alpha_j x\|^2 +\eps^2|\alpha_j|^2
    -2\eps(1-\eps)\real\scp{M_0 x-\alpha_j x,\alpha_j x}\\
\geq\,& (1-\eps)^2r_j^2 +\eps^2|a_j|^2 -2\eps(1-\eps)\big(d_0^2 -|\alpha_j|^2\big)\\
>\,&\big((1-\eps)^2+\eps^2+2\eps(1-\eps)\big)r_j^2=r_j^2,
\end{align*}
where the last inequality follows from \eqref{holes-ring}.

Take $l=1$. Knowing that $\|M_1 x\|\geq d_0$ ($x\in\cH_1$) one can obtain
\begin{align*}
\|S_{1,\eps}x-\alpha_jx\|^2 =\,&\|M_1 x-\alpha_j x+\eps M_1 x\|^2 \\
=\, & \|M_1 x-\alpha_j x\|^2 +\eps^2 \|M_1 x\|^2
-2\eps\real\scp{M_1 x,\alpha_j x} +2\eps\|M_1 x\|^2\\
\geq\, & r_j^2+\eps^2\|M_1 x\|^2 +\eps(\|M_1 x\|^2-d_0^2)\\
\geq\, & r_j^2+\eps^2d_0^2 > r_j^2.
\end{align*}

Finally, take $l=2$. We know that
$r_j+|\alpha_j|<(1-\eps)d_1 \leq (1-\eps)\|M_2 x\|$.
Hence
\begin{align*}
\|S_{2,\eps} x-\alpha_j x\| &\geq \|(1-\eps)M_2 x\|-\|\alpha_jx\|\\
&    \geq (1-\eps)d_1-|\alpha_j| > r_j.
\end{align*}
This ends the  proof of \eqref{frombelow} and the theorem.
\end{proof}

\begin{corollary}\label{cor:decentered}
For an  annulus $\Omega$ which is possibly decentered (arbitrary inner and outer radii and centers of the circles,
as long as the hole is contained in the outer disc), i.e. $k_1=1$, $k_2=k=2$ and $|\alpha_1-\alpha_2|+r_2<r_1$,
the Agler norms $\norm{\cdot}_{\cT_\Omega}$ and $\norm{\cdot}_{\cT_\Omega^\circ}$ coincide.
\end{corollary}
\begin{proof}
The case not covered by (ii) is when $|\alpha_2|<r_2<r_1-|\alpha_2|$.
The proof is  analogous, however we need to consider
$|T|=M_1'\oplus M_2'$ with $\sigma(M_1')\subset[0,d_1]$ and
define $T_\eps = (1-\eps)UM_2'\oplus (1+\eps)UM_1'$. 
\end{proof}

\begin{remark}
A similar construction can be applied to other domains. Namely, it is possible to add another `rings of holes' as long as they are 
separated from the existing one by an annulus of nonzero width.
\end{remark}

\section{Special cases}\label{sec:annulus}

\subsection{The convex case}
   Observe that if  $\Omega$ is convex, it contains the numerical range.
Recall the seminal result  \cite{CP}, that the complete spectral constant is bounded by $1+\sqrt 2$.
Therefore, we have the following corollary.
  
\begin{corollary}
Let $\Omega$ given by \eqref{Omegadef} be convex. For any $F\in \cM_n( \mathcal{R}(\Omega) )$ with
$$\sup_{z\in \Omega}\|F(z)\| \leq(1+\sqrt 2)^{-1},$$
there exists a positive integer $m$ and a contractive matrix
{\tiny $\begin{bmatrix} A & B \\ C & D \end{bmatrix}$} of size $(km+n) \times (km+n)$ such that
\begin{equation}
F(z) = D +C\mathbf{P}_-(z)_{m}(\mathbf{P}_+(z)_{m}-A\mathbf{P}_-(z)_{m})^{-1}B, 
\end{equation}
where $\mathbf P_\pm(z)$ are defined by \eqref{PpD} and \eqref{PmD}.
\end{corollary}

\subsection{The case of the annulus}

Let 
$$
A_{R,r} = \{z \in \mathbb{C}: r \leq |z| \leq R\}, \quad 0 < r < R
$$
be the annulus. Recall the formulas
\[
\Pp(z)=\begin{bmatrix}
R & 0\\ 0 & z\end{bmatrix},\quad \Pm(z)=\begin{bmatrix}
z & 0\\ 0 & r\end{bmatrix}
\]
and for a positive integer $k$ we set ${\bf P}_\pm(z)_{k}:={\bf P}_\pm(z)\otimes I_{k}$.
In this setting the map $\gamma: A_{R,r} \to \overline{\mathbb D^2}$ takes the form
$$\gamma(z) = \left(\frac{z}{R}, \frac{r}{z}\right).$$
Theorem \ref{preptheoremD} gives us the following corollary, cf.\ Mittal \cite{Mittal} for an analogous result
for $H^{\infty}$ functions supported by an extensive theory.
\begin{corollary} \label{cor_annulus}
For all $ F \in \cM_n(\mathcal{R}(A_{R,r})) $ we have the following identity for the Agler norm on the annulus:
\begin{equation}\label{Mittalrat}
\|F\|_{\mathcal T_{A_{R,r}}} = \inf \left \{\sup_{z\in{A_{R,r}}} \|G(z)\|:
	G \in \cM_n(\mathcal{R}(\overline{\mathbb{D}^2})), \gamma^*(G) = F\right\}.
\end{equation}
\end{corollary}
\begin{proof}
For $F \in \cM_n(\mathcal{R}(A_{R,r}))$, due to Theorem \ref{preptheoremD}, we have
$$
\|F\|_{\mathcal T\ker\gamma^*} =  \|F\|_{\mathcal T_{A_{R,r}}}
$$
According to Ando's theorem, this implies \eqref{Mittalrat}. 
\end{proof}

It is natural to ask for which functions the  infimum in \eqref{Mittalrat} is attained.
While in general it is not yet answered, we discuss some special cases. For clarity, in the following proposition we restrict
our attention to $A_{1, r}$. It is worth noting that using the well known fact, that two annuli are conformally equivalent
if and only if their radii have the same ratio. Hence, considerations for $A_{R, r}$ are equivalent
to those for $A_{1,\frac{r}{R}}$. 
\begin{proposition}\label{k,l}
Let $F(z) = z^k + \frac{1}{z^l}$ where $k$ and $l$ are positive integers and let
$G(z_1, z_2) = z_1^k + \frac{z_2^l}{r^l}$. Then 
$$
\|F\|_{\mathcal T_{A_{1,r}}}=\norm G_\infty: =\sup_{z\in\overline{\mathbb D^2}} \|G(z)\|
$$
\end{proposition}

\begin{proof}
Since $\gamma^*(G)=F$ we clearly have $\|F\|_{\mathcal T_{A_{1,r}}} \leq \|G\|_{\infty}$. To prove the equality we find a matrix $M$ such that $M \in \mathcal T_{A_{1,r}}$ and $\|F(M)\| = \|G\|_{\infty}$. Note that
$$|G(z_1, z_2)| = \left|z_1^k + \frac{z_2^l}{r^l}\right| \leq 1 + \frac{1}{r^l}$$
and for $z_1 = z_2 = 1$ we have the equality so $\|G(z_1, z_2)\|_{\infty} = 1 + \frac{1}{r^l}$.
Now let
$$
M = \begin{bmatrix}
0& 1 & \\
& \ddots& \ddots \\
& & 0& 1 \\
& & &0 & r \\
& & & & \ddots& \ddots & \\
& & & & &0 & r \\ 
r & & & &&& 0
\end{bmatrix},
$$
i.e., it is a $(k+l) \times (k+l)$ matrix with $0$ entries except for $1$ in entries
$(i, i+1)$ for $i = 1, \ldots, k$ and $r$ in entries $(i, i+1)$ for $i = k+1, \ldots, k+l-1$ and $r$ in entry  $(k+l, 1)$.
Since  $r < 1$, it follows that $\|M\| = 1$ and $\|M^{-1}\| = \frac{1}{r}$. One may observe that 
$$
F(M) = M^k + M^{-l} = \begin{bmatrix}
& & & a_1 & \\
& & & & \ddots \\
& & & & & a_l \\
a_{l+1} & & & & & \\ 
& \ddots & & & & \\
& & a_{k+l} 
\end{bmatrix},$$
where $a_i = r^{i-1} + \frac{1}{r^{l+1-i}}$ for $i = 1, \ldots, \min(k,l)$, $a_i = r^{\min(k,l)} + \frac{1}{r^{l-\min(k,l)}}$
for $i = \min(k,l)+1, \ldots, \max(k,l)+1$ and $a_i = r^{k+l+1-i} + \frac{1}{r^{i-k-1}}$ for $i = \max(k,l)+2, \ldots, k+l$.
It follows that $\max(a_1, \ldots, a_{k+l}) = a_1 = 1 + \frac{1}{r^l}$.
Therefore $\|F(M)\| =  1 + \frac{1}{r^l} = \|G\|_{\infty}$ and we obtain the claim.
\end{proof}

The above Proposition shows that for the function $F(z) = z^k + \frac{1}{z^l}$ the infimum in \eqref{Mittalrat}
is attained by a function $G$ satisfying  $\gamma^*(G)=F$, constructed according to formula \eqref{FtoG}.
However, this is not always the case, as the following example shows.

\begin{example}
Let $F(z) = z + \frac{1}{z} + 1$, when $R=1$ and $r=\frac{1}{2}$. Formula \eqref{FtoG} produces a function
$G(z_1, z_2) = z_1+2z_2+1$ that has the supremum norm on the bidisc equal to $4$. However, one may calculate
that the function $G_1(z_1, z_2) = z_1+2z_2+1-\frac{z_1z_2-\frac{1}{2}}{2}$, which satisfies $\gamma^*(G_1)=F$,
has the supremum norm strictly less than $4$, showing that $\|F\|_{\mathcal T_{A_{1,r}}} < \|G\|_{\infty}$.
\end{example}

Theorem \ref{mainT} gives us an immediate corollary, namely the finite dimensional realization theorem for the annulus.
Note that for $\Omega$ being the annulus Propostion~\ref{basicAglerD} says that $A_{R,r}$
is a $\left(2+\frac{2\sqrt 3}3 \right) $-complete spectral set for any $T\in\mathcal T_{A_R,r}$ and any $R>r$.
In the theorem below we use the   inverse of this explicit value $\left( 2+\frac{2\sqrt 3}3\right)^{-1}= \frac{3-\sqrt 3}4$.
Note that that there exists other estimates on the complete spectral constant for the annulus. In particular, 
any bound independent on $T$ has to be greater or equal to 2  \cite{tsikalas2021}. 

\begin{corollary} \label{realizationTannulus}
For any $F\in \cM_n( \mathcal{R}(A_{R,r}) )$ with  $\sup_{r\leq|z|\leq R}\|F(z)\|< \frac{3-\sqrt 3}4$
there exists a positive integer $m$ and a contractive matrix
{\tiny $\begin{bmatrix} A & B \\ C & D \end{bmatrix}$} of size $(2m+n) \times (2m+n)$
such that
\begin{equation*}\label{1}
F(z) = D +C\mathbf{P}_-(z)_{m}(\mathbf{P}_+(z)_{m}-A\mathbf{P}_-(z)_{m})^{-1}B.
\end{equation*}
\end{corollary}

We end this section with a connection to systems theory.

\begin{example}
Consider the system of difference equations
$$ (\Sigma ) \ \ \left\{ \begin{matrix} 
\begin{pmatrix} x_{k+1} \cr \tilde{x}_k \end{pmatrix} = A \begin{pmatrix} x_{k} \cr r\tilde{x}_{k+1} \end{pmatrix} +Bu_k & k=\dots , -1,0,1 , \dots, \cr y_k =C \begin{pmatrix} x_{k} \cr r\tilde{x}_{k+1} \end{pmatrix} + Du_k & k=\dots , -1,0,1 , \dots . \end{matrix} \right.$$
Here $(u_k)_{k\in{\mathbb Z}}$ is the input sequence, $(y_k)_{k\in{\mathbb Z}}$ the output sequence, and $\begin{pmatrix} x_{k} \cr \tilde{x}_{k} \end{pmatrix}_{k\in{\mathbb Z}}$ represents the state space variable. Let now 
$$ u(z) = \sum_{k=-\infty}^\infty u_kz^k , \ 
y(z) = \sum_{k=-\infty}^\infty y_kz^k ,\ 
\begin{pmatrix} x(z) \cr \tilde{x}(z) \end{pmatrix} = \sum_{k=-\infty}^\infty \begin{pmatrix} x_{k} \cr \tilde{x}_{k} \end{pmatrix}z^k ,
$$ which we assume to be convergent for $r<|z|<1$.
Multiplying the equations in the system $(\Sigma)$ by $z^k$ and summing them, we obtain the following
$$  \left\{ \begin{matrix} \begin{pmatrix} \frac1z & 0 \cr 0 & 1 \end{pmatrix} \begin{pmatrix} x(z) \cr \tilde{x}(z) \end{pmatrix} = A \begin{pmatrix} 1 & 0 \cr 0 & \frac{r}{z} \end{pmatrix}  \begin{pmatrix} x(z) \cr \tilde{x}(z) \end{pmatrix} +Bu(z) \cr\ \ \ \ \ \ \ \ \ \  y(z) =C \begin{pmatrix} 1 & 0 \cr 0 & \frac{r}{z} \end{pmatrix}  \begin{pmatrix} x(z) \cr \tilde{x}(z) \end{pmatrix} + Du(z)  . \end{matrix} \right.$$
Using the first equation to solve for $\begin{pmatrix} x(z) \cr \tilde{x}(z) \end{pmatrix}$ and plugging it into the second equation, we obtain the following relation between the input $u(z)$ and output $y(z)$:
$$ y(z) = \left[ C \begin{pmatrix} 1 & 0 \cr 0 & \frac{r}{z} \end{pmatrix} \left( \begin{pmatrix} \frac1z & 0 \cr 0 & 1 \end{pmatrix} - A \begin{pmatrix} 1 & 0 \cr 0 & \frac{r}{z} \end{pmatrix} \right)^{-1} B + D\right] u(z) $$
$$ \ \ \ \ \ \ \ \ \ \ \ \ = \left[ C \begin{pmatrix} z & 0 \cr 0 & r \end{pmatrix} \left( \begin{pmatrix} 1 & 0 \cr 0 & z \end{pmatrix} - A \begin{pmatrix} z & 0 \cr 0 & {r} \end{pmatrix} \right)^{-1} B + D\right] u(z)
=: F(z) u(z) . $$
Note that $F(z)$ has exactly the form as in Corollary \ref{realizationTannulus}.
\end{example}

\section{The Bohr inequality for the annulus and the bidisc}
\label{sec:bohr}

In this section we consider scalar-valued functions only.  We will define the Bohr constant for the annulus and show its relation with the known Bohr  constant for the bidisc and improve the upper bound on the latter one. Let us review the known definitions and  facts.        For a function $f(z)=\sum_{\alpha\in\mathbb N^d}c_\alpha z^\alpha$,  holomorphic on the polydisc we define
$$
\norm f_{\hat L^1}:=\sum_\alpha |c_\alpha|.
$$
Further, the constant $K_d$ is defined as 
$$
K_d=\sup\left\{ \rho>0:   \norm{ f_\rho}_{\hat L^1} \leq 1 \text{ for all } f  \in H^\infty(\mathbb D^d)\text{ such that } \sup_{z\in\mathbb D^d}| f(z)| \leq 1 \right\},
$$
where $f_\rho(z)=f(z\rho)$, $\rho>0$. Several results were establishing the estimates of the constants $K_d$; we refer to a recent paper by Knese \cite{knese2024three} for an elegant summary. Furthermore, Knese in  \cite{knese2024three} defines and gives estimates for another sequence of constants
$$
K_d( \mathcal T_{\mathbb D^d} ) =\sup\left\{ \rho>0:   \norm{ f_\rho}_{\hat L^1} \leq 1 \text{ for all } f\in H^\infty(\mathbb D^d)   \text{ such that } \| f\|_{\mathcal T^\circ_{\mathbb D^d}}\leq 1 \right\}.
$$ 
 Note that $\| f\|_{\mathcal T^\circ_{\mathbb D^d}}$ in the above formula denotes the Agler norm with respect to the set of tuples of strict Hilbert space contractions, see \eqref{Aglercirc}, which is well defined  for $f\in H^\infty(\mathbb D^d) $.  
  
Our interest will lie only in   $d=1,2$, for which we have the following:
$$
\frac 13=K_1=K_1(\mathcal T_{\mathbb D} ) \geq K_2=K_2(\mathcal T_{\mathbb D^2} )\geq 0.3006.
$$
The first equality above is due to Bohr \cite{Bohr14}, the second expresses  von Neumann inequality, the  inequality $K_1\geq K_2$ is obvious, while the last equality follows from the Ando's theorem; finally the last inequality is due to Knese  \cite{knese2024three}.

To keep the analogy, for  $f\in H^\infty(A_{R,r})$ ($R>r$) having the expansion  $f= \sum_{k=-\infty}^\infty  a_k z^k$  we introduce the  norm
$$
\norm{ f }_{\hat L^1,\rho }  =  \sum_{k\geq 0}  |a_k|R^k\rho^k  +\sum_{k<0}  |a_k| \rho^{-k} r^{k} , \quad \rho>0
$$
and define the constants
$$
K_1(A_{R,r}):=\sup\left\{ \rho>0:   \norm{f}_{\hat L^1,\rho} \leq 1 \text{ for all }f \in H^\infty(A_{R,r}),\ \sup_{z\in A_{R,r}} | f(z)| \leq 1 \right\},
$$
and
\begin{equation}\label{Hinfsup}
K_1(\mathcal T_{A_{R,r}}):=\sup\left\{ \rho>0:   \norm{f}_{\hat L^1,\rho}\leq 1 \text{ for all }f \in H^\infty(A_{R,r}),\  \| f\|_{\mathcal T^\circ_{A_{R,r}}}\leq 1 \right\}.
\end{equation}
 Note that, as in the polydisc case, the Agler norm $ \| f\|_{\mathcal T^\circ_{A_{R,r}}}$ is well defined; see \eqref{Aglercirc} for notation.  Further, 
one may replace $H^\infty(A_{R,r})$ by $\mathcal R(A_{R,r})$.
\begin{lemma}\label{KR} The following holds
 \begin{equation}\label{Rsup}
K_1(\mathcal T_{A_{R,r}})=\sup\left\{ \rho>0:   \norm{f}_{\hat L^1,\rho}\leq 1 \text{ for all }f \in \mathcal R(A_{R,r}),\  \| f\|_{\mathcal T_{A_{R,r}}}\leq 1 \right\}.
\end{equation}
\end{lemma} 
\begin{proof}
The statement has the form
$$
\sup W_\mathcal{R} =  \sup W_{H^\infty},
$$
where the subsets   $W_{H^\infty}$ and $W_{\mathcal R}$ of $(0,+\infty)$ are defined according to \eqref{Hinfsup} and \eqref{Rsup}, respectively.
We show that $  W_{H^\infty}= W_\mathcal{R}$. The inclusion `$\subseteq$' follows directly from  Proposition~\ref{Agler<}.
To see the  reverse inclusion  take  $\rho>0$ which is not in $W_{H^\infty}$. Hence,   there exists 
    $f \in H^\infty(A_{R,r})$, $f(z)=\sum_{k=-\infty}^\infty a_kz^k$
    with $\norm f_{\mathcal T^\circ_{A(R,r)}}\leq 1$ and $\norm f_{\hat L,\rho}>1$. 
    The related  Laurent polynomials 
$$
f_n(z)=\sum_{k=-2n}^{2n}\big(1-\frac{|k|}{2n+1}\big)   a_kz^k
$$ 
converge locally uniformly to $f$ with $\norm {f_n}_{\mathcal T^\circ_{A(R,r)}}\leq \norm f_{\mathcal T^\circ_{A(R,r)}}$,
cf.\ \cite[Theorem 4.2.10]{Mittal}.   Further,  it is elementary that  $\norm{ f_n}_{\hat L^1,\rho} $ converges with $n$
to $\norm{ f}_{\hat L^1,\rho}>1 $.  Hence, for $n$ sufficiently large,  the function $f_n$ witnesses that $\rho$
is not in $W_\mathcal{R}$, which finishes the proof.
\end{proof}

It is apparent that 
$$
K_1(A_{R,r})\leq K_1(\mathcal T_{A_{R,r}})\leq \frac 13,\quad 0<r<R. 
$$
Now let us show a lower bound on $K_1(\mathcal T_{A_{R,r}})$. 
\begin{theorem}\label{K1K2}
If $r/R\leq K_2^2$, then  $K_1(\mathcal T_{A_{R,r}})\geq K_2$.
\end{theorem}

\begin{proof}
Fix arbitrary $f(z)=\sum_{k=-\infty}^\infty  a_k z^k$ with  $\| f\|_{\mathcal T_{A_{R,r}}}< 1 $ and $\rho<K_2$.
By Lemma~\ref{KR} we may assume that $f\in\mathcal R(A_{R,r})$. 
By Theorem \ref{preptheoremD} there exists $g\in\mathcal R(\mathbb D^2)$ of norm
$\norm g_{\infty,\mathbb D^2}=\norm g_{\mathcal T_{\mathbb D^2}}<1$  such that $g\circ \gamma =f$.
Let $g(z)=\sum_\alpha c_\alpha z^\alpha$ be the Taylor expansion, where $z=(z_1,z_2)$, $\alpha=(\alpha_1,\alpha_2)$.
By definition of $K_2$ we have that $\sum_\alpha{c_\alpha} \rho^{|\alpha|}\leq 1$. Further
$$
f(z)=g\big(\frac zR,\frac rz\big )= \sum c_\alpha  \big(\frac zR\big)^{\alpha_1} \big(\frac rz\big)^{\alpha_2}=\sum c_\alpha R^{-\alpha_1}r^{\alpha_2}   z^{\alpha_1-\alpha_2}.
$$
Comparing this with the original expansion of $f$ we see that
$$
a_k=\sum_{\alpha_1-\alpha_2=k} c_\alpha \frac{ r^{\alpha_2}}{R^{\alpha_1}},\quad k \in\mathbb Z 
$$
and so
$$
s_+:=\sum_{k\geq0} |a_k|R^k\rho^{k}\leq \sum_{k\geq0}\sum_{\alpha: \alpha_1-\alpha_2=k} |c_\alpha |\frac{ r^{\alpha_2}}{R^{\alpha_1}} R^k \rho^{k}=\sum_{\alpha: \alpha_1\geq\alpha_2} | c_\alpha|  \left( \frac r{R\rho}\right)^{\alpha_2} \rho^{\alpha_1}
$$
while  
$$
s_-:=\sum_{k<0} |a_k|r^{k}\rho^{-k}\leq 
\sum_{k<0}\sum_{\alpha: \alpha_1-\alpha_2=k}  |c_\alpha| \frac{r^{\alpha_2}}{R^{\alpha_1}} r^k \rho^{-k}
	= \sum_{\alpha:\alpha_1<\alpha_2} |c_\alpha| \left( \frac r{R\rho}\right)^{\alpha_1}  \rho^{\alpha_2}.
$$
As $r/R<K_2^2$ the interval $(\frac r{RK_2},K_2)$ is nonempty and  for all $\rho\in(\frac r{RK_2},K_2)$ one obtains
$s_++s_-\leq \sum_\alpha |c_\alpha| s^{\alpha_1+\alpha_2}<1$ with  $s=\max(\rho, \frac r{R\rho})<K_2$.
Consequently, $K_1(\mathcal T_{A_{R,r}})\geq K_2$.
\end{proof}

As discussed in the Introduction, the above result has a consequence in the theory of spectral constants in Banach algebras.
We refer to \cite{BHSV2025} for related  results on the disc and numerical range.  The constant $\Psi(A(R,r))$
appearing below is the spectral constant of the annulus, it is known to be less or equal to $(1+\sqrt 2)$
\cite[Theorem 10]{crouzeix2019spectral}, and not smaller than 2 \cite{tsikalas2021}.
See also \cite{crouzeix2012annulus} for other $R$-dependent bounds.

\begin{theorem} \label{Banach}
Consider a Banach  algebra $\mathcal B$. If $T\in \mathcal{B}$ is such that $\norm T_{\mathcal B}< R$ and
$\norm{T^{-1}}_{\mathcal B}<  r^{-1}$,  then the annulus $A(\tilde R,\tilde r)$, $\tilde R:=K_2^{-1}R$,
$\tilde r:=K_2r$ is a $\Psi(A(R,r))$-spectral set for $T$, namely  
$$
\norm{f(T)}_{\mathcal B}\leq   \| f \|_{\mathcal T(A_{\tilde R,\tilde r    })}\leq \Psi(A(R,r))\sup_{\tilde r\leq |z|\leq \tilde R} |f(z)| 
$$
for any $ f\in\mathcal R( \tilde R,\tilde r )$.
\end{theorem}

\begin{proof}
Observe that $\tilde r/\tilde R<K_2^2$, hence, by Theorem~\ref{K1K2} we have $ K_1(\mathcal T_{\tilde R,\tilde r} )\geq K_2$.
Thus, by definition of $K_1(\mathcal T_{\tilde R,\tilde r} )$,
for $f(z)=\sum_{k=-\infty}^\infty a_kz^k \in\mathcal R( A_{K_2^{-1}R,K_2r })$ with $\|f\|_{\mathcal T(A_{R,r})}=1 $
and some $\rho <K_2\leq K_1(\mathcal T_{\tilde R,\tilde r} )$ we have
\begin{eqnarray*}
\norm{f(T)}_{\mathcal B} &\leq& \sum_{k\geq0} |a_k|\norm{T}_{\mathcal B}^k +\sum_{k<0}|a_k|\norm{T^{-1}}_{\mathcal B}^{-k} \\
&\leq& \sum_{k\geq0} |a_k|\tilde R^k\rho ^k  +\sum_{k<0}|a_k|  \tilde r^{k}\rho^{-k} \\
&\leq & \norm {f}_{\hat L^1,_\rho}\\
&\leq & 1,
\end{eqnarray*}
which shows the first inequality. The second follows from the definition of $\Psi(A(R,r))$.
\end{proof}

It is known that $K_2\in[0.3006, 1/3  ]$. We use this opportunity to provide a better upper bound by giving a concrete
rational function $f$. The values of the coefficients were found by a numerical optimizing procedure,
but the proof uses exact calculations.

\begin{theorem}\label{K2est} $K_2<0.3177$. \end{theorem}

\begin{proof}
It is enough to find a function $f(z)=\sum_{\alpha}  c_\alpha z^\alpha$, $z=(z_1,z_2)$, analytic on the bidisc,
with supremum norm one, such that
\begin{equation}\label{rho>1}
\sum_{\alpha_1,\alpha_2=0}^{12} |c_\alpha| \rho_0^{|\alpha|}>1,\quad \text{where } \rho_0:=0.3177. 
\end{equation}
The function $f$ is of the following form:
$$
f(z_1,z_2)=\frac{p(z)}{\bar p(z)}= \frac{z_1z_2+a_{01}z_2+a_{10}z_1+a_{00}  }{ a_{00}z_1z_2+a_{10}z_2+a_{01}z_1+1  },
$$
where
$$
p(z)=\det \left( I_2-D\begin{bmatrix} z_1\\& z_2\end{bmatrix}\right), \quad D=\begin{bmatrix}  0.854373111798292 &  -0.518782521594128  \\ 0.518794363700548   & 0.848187547437653
\end{bmatrix}.
$$
Observe that $\norm D<1$, hence $f$ is a rational inner function, with singularities outside the closed unit bidisc,
cf., e.g., \cite{knese2024rational}. 
The values of $c_{\alpha}$ ($|\alpha|\leq 12$) were computed in Python  in exact arithmetic using SymPy \cite{sympy},
as a solution of a linear equation arising from the polynomial equation
$$
\left(\sum_{\alpha_1,\alpha_2=0}^{12} c_\alpha z^\alpha\right)\bar p(z)=p(z).
$$
The same software was used to verify that \eqref{rho>1} holds.
\end{proof}

\section*{Acknowledgments}
H. J. Woerdeman and M. Wojtylak wish to thank the Banff International Research Station for the opportunity
to work on this project during the workshop 23w5056.
The research of H. J. Woerdeman is partially supported by NSF grant DMS 2348720.
R. Baran was  supported by the  ID.UJ Initiative of Excellence - Research University.

The authors would like to thank Greg Knese for helpful discussions on the topic of the Bohr radius.

\bibliographystyle{plain}
\bibliography{refs}

\begin{thebibliography}{10}

\bibitem{Agler}
Jim Agler.
\newblock On the representation of certain holomorphic functions defined on a
  polydisc.
\newblock In {\em Topics in operator theory: {E}rnst {D}. {H}ellinger memorial
  volume}, volume~48 of {\em Oper. Theory Adv. Appl.}, pages 47--66.
  Birkh\"auser, Basel, 1990.

\bibitem{agler2023complete}
Jim Agler, {\L}ukasz Kosi{\'n}ski, and John~E McCarthy.
\newblock Complete norm-preserving extensions of holomorphic functions.
\newblock {\em Israel Journal of Mathematics}, 255(1):251--263, 2023.

\bibitem{AglerMcCarthy}
Jim Agler and John~E. McCarthy.
\newblock {\em Pick interpolation and {H}ilbert function spaces}, volume~44 of
  {\em Graduate Studies in Mathematics}.
\newblock American Mathematical Society, Providence, RI, 2002.

\bibitem{AT}
C.-G. Ambrozie and D.~Timotin.
\newblock A von {N}eumann type inequality for certain domains in {$\mathbf
  C^n$}.
\newblock {\em Proc. Amer. Math. Soc.}, 131(3):859--869, 2003.

\bibitem{Ando}
T.~And\^o.
\newblock On a pair of commutative contractions.
\newblock {\em Acta Sci. Math. (Szeged)}, 24:88--90, 1963.

\bibitem{Arov}
D.~Z. Arov.
\newblock Passive linear steady-state dynamical systems.
\newblock {\em Sibirsk. Mat. Zh.}, 20(2):211--228, 457, 1979.

\bibitem{BBC2009}
Catalin Badea, Bernhard Beckermann, and Michel Crouzeix.
\newblock Intersections of several disks of the {R}iemann sphere as
  {$K$}-spectral sets.
\newblock {\em Commun. Pure Appl. Anal.}, 8(1):37--54, 2009.

\bibitem{BB}
Joseph~A. Ball and Vladimir Bolotnikov.
\newblock Canonical transfer-function realization for {S}chur-{A}gler-class
  functions on domains with matrix polynomial defining function in {$\mathbb
  C^n$}.
\newblock In {\em Recent progress in operator theory and its applications},
  volume 220 of {\em Oper. Theory Adv. Appl.}, pages 23--55.
  Birkh\"auser/Springer Basel AG, Basel, 2012.

\bibitem{BallCohen}
Joseph~A. Ball and Nir Cohen.
\newblock de {B}ranges-{R}ovnyak operator models and systems theory: a survey.
\newblock In {\em Topics in matrix and operator theory ({R}otterdam, 1989)},
  volume~50 of {\em Oper. Theory Adv. Appl.}, pages 93--136. Birkh\"auser,
  Basel, 1991.

\bibitem{BGR}
Joseph~A. Ball, Israel Gohberg, and Leiba Rodman.
\newblock {\em Interpolation of rational matrix functions}, volume~45 of {\em
  Operator Theory: Advances and Applications}.
\newblock Birkh\"auser Verlag, Basel, 1990.

\bibitem{BGK}
Harm Bart, Israel Gohberg, Marinus~A. Kaashoek, and Andr\'e{} C.~M. Ran.
\newblock {\em Factorization of matrix and operator functions: the state space
  method}, volume 178 of {\em Operator Theory: Advances and Applications}.
\newblock Birkh\"auser Verlag, Basel, 2008.
\newblock Linear Operators and Linear Systems.

\bibitem{BHSV2025}
Hanna Blazhko, Daniil Homza, Felix~L. Schwenninger, Jens de~Vries, and Michał
  Wojtylak.
\newblock The algebraic numerical range as a spectral set in {B}anach algebras.
\newblock {\em Canadian Journal of Mathematics}, page 1–25, 2025.

\bibitem{Blecher1990}
David~P. Blecher, Zhong-Jin Ruan, and Allan~M. Sinclair.
\newblock A characterization of operator algebras.
\newblock {\em J. Funct. Anal.}, 89(1):188--201, 1990.

\bibitem{Bohr14}
Harald Bohr.
\newblock {A Theorem Concerning Power Series}.
\newblock {\em Proceedings of the London Mathematical Society}, s2-13(1):1--5,
  01 1914.

\bibitem{Breh61}
S.~Brehmer.
\newblock {\"U}ber vertauschbare {K}ontraktionen des {H}ilbertschen {R}aumes.
\newblock {\em Acta Sci.\ Math.\ (Szeged)}, 22:106--111, 1961.

\bibitem{CP}
M.~Crouzeix and C.~Palencia.
\newblock The numerical range is a {$(1+\sqrt{2})$}-spectral set.
\newblock {\em SIAM J. Matrix Anal. Appl.}, 38(2):649--655, 2017.

\bibitem{crouzeix2012annulus}
Michel Crouzeix.
\newblock The annulus as a k-spectral set.
\newblock {\em Commun. Pure Appl. Anal}, 11(6):2291--2303, 2012.

\bibitem{crouzeix2019spectral}
Michel Crouzeix and Anne Greenbaum.
\newblock Spectral sets: numerical range and beyond.
\newblock {\em SIAM Journal on Matrix Analysis and Applications},
  40(3):1087--1101, 2019.

\bibitem{dBR1}
Louis de~Branges and James Rovnyak.
\newblock Canonical models in quantum scattering theory.
\newblock In {\em Perturbation {T}heory and its {A}pplications in {Q}uantum
  {M}echanics ({P}roc. {A}dv. {S}em. {M}ath. {R}es. {C}enter, {U}.{S}. {A}rmy,
  {T}heoret. {C}hem. {I}nst., {U}niv. of {W}isconsin, {M}adison, {W}is.,
  1965)}, pages 295--392. Wiley, New York-London-Sydney, 1966.

\bibitem{dBR2}
Louis de~Branges and James Rovnyak.
\newblock {\em Square summable power series}.
\newblock Holt, Rinehart and Winston, New York-Toronto-London, 1966.

\bibitem{grinshpan2016matrixvaluedhermitianpositivstellensatzlurking}
Anatolii Grinshpan, Dmitry~S. Kaliuzhnyi-Verbovetskyi, Victor Vinnikov, and
  Hugo~J. Woerdeman.
\newblock Matrix-valued {H}ermitian {P}ositivstellensatz, lurking contractions,
  and contractive determinantal representations of stable polynomials.
\newblock In {\em Operator theory, function spaces, and applications}, volume
  255 of {\em Oper. Theory Adv. Appl.}, pages 123--136. Birkh\"auser/Springer,
  Cham, 2016.

\bibitem{GKVW2016}
Anatolii Grinshpan, Dmitry~S. Kaliuzhnyi-Verbovetskyi, Victor Vinnikov, and
  Hugo~J. Woerdeman.
\newblock Stable and real-zero polynomials in two variables.
\newblock {\em Multidimens. Syst. Signal Process.}, 27(1):1--26, 2016.

\bibitem{GKW2013}
Anatolii Grinshpan, Dmitry~S. Kaliuzhnyi-Verbovetskyi, Victor Vinnikov, and
  Hugo~J. Woerdeman.
\newblock Stable and real-zero polynomials in two variables.
\newblock {\em Multidimens. Syst. Signal Process.}, 27(1):1--26, 2016.

\bibitem{Woerdeman2012}
Anatolii Grinshpan, Dmitry~S. Kaliuzhnyi-Verbovetskyi, and Hugo~J. Woerdeman.
\newblock Norm-constrained determinantal representations of multivariable
  polynomials.
\newblock {\em Complex Anal. Oper. Theory}, 7(3):635--654, 2013.

\bibitem{JW}
Joshua~D. Jackson and Hugo~J. Woerdeman.
\newblock Minimal realizations and determinantal representations in the
  indefinite setting.
\newblock {\em Integral Equations Operator Theory}, 94(2):Paper No. 18, 13,
  2022.

\bibitem{Jackson}
Joshua~Dunne Jackson.
\newblock {\em Minimal {R}ealizations and {D}eterminantal {R}epresentations in
  the {I}ndefinite {S}etting}.
\newblock ProQuest LLC, Ann Arbor, MI, 2020.
\newblock Thesis (Ph.D.)--Drexel University.

\bibitem{jury2023positivity}
Michael~T Jury and Georgios Tsikalas.
\newblock Positivity conditions on the annulus via the double-layer potential
  kernel.
\newblock arXiv:2307.13387.

\bibitem{KM66}
VE~Katsnelson and VI~Matsaev.
\newblock Spectral sets for operators in a {B}anach space and estimates of
  functions of finite-dimensional operators (in russian).
\newblock {\em Teor. Funkcii Funkcional. Anal. i Prilozen.}, 3:3--10, 1966.

\bibitem{Knese2021}
Greg Knese.
\newblock Kummert's approach to realization on the bidisk.
\newblock {\em Indiana Univ. Math. J.}, 70(6):2369--2403, 2021.

\bibitem{knese2024rational}
Greg Knese.
\newblock Rational inner functions on the polydisk--a survey.
\newblock arXiv:2409.14604.

\bibitem{knese2024three}
Greg Knese.
\newblock Three radii associated to {S}chur functions on the polydisk.
\newblock arXiv:2410.21693.

\bibitem{Kummert}
Anton Kummert.
\newblock Synthesis of two-dimensional lossless {$m$}-ports with prescribed
  scattering matrix.
\newblock {\em Circuits Systems Signal Process.}, 8(1):97--119, 1989.

\bibitem{mccullough2023geometric}
Scott McCullough and James~E Pascoe.
\newblock Geometric dilations and operator annuli.
\newblock {\em Journal of Functional Analysis}, 285(7):110035, 2023.

\bibitem{sympy}
Aaron Meurer, Christopher~P. Smith, Mateusz Paprocki, Ond\v{r}ej
  \v{C}ert\'{i}k, Sergey~B. Kirpichev, Matthew Rocklin, AMiT Kumar, Sergiu
  Ivanov, Jason~K. Moore, Sartaj Singh, Thilina Rathnayake, Sean Vig, Brian~E.
  Granger, Richard~P. Muller, Francesco Bonazzi, Harsh Gupta, Shivam Vats,
  Fredrik Johansson, Fabian Pedregosa, Matthew~J. Curry, Andy~R. Terrel,
  \v{S}t\v{e}p\'{a}n Rou\v{c}ka, Ashutosh Saboo, Isuru Fernando, Sumith Kulal,
  Robert Cimrman, and Anthony Scopatz.
\newblock Sympy: symbolic computing in {P}ython.
\newblock {\em PeerJ Computer Science}, 3:e103, January 2017.

\bibitem{Mittal}
Meghna Mittal.
\newblock {\em Function theory on the quantum annulus and other domains}.
\newblock ProQuest LLC, Ann Arbor, MI, 2010.
\newblock Thesis (Ph.D.)--University of Houston.

\bibitem{MittalPaulsen}
Meghna Mittal and Vern Paulsen.
\newblock Operator algebras of functions.
\newblock {\em Journal of Functional Analysis}, 258:3195--3225, 05 2010.

\bibitem{paulsen2022}
Vern~I. Paulsen and Dinesh Singh.
\newblock A simple proof of {B}ohr's inequality, arXiv:2201.10251.

\bibitem{tsikalas2021}
Georgios Tsikalas.
\newblock A note on a spectral constant associated with an annulus.
\newblock {\em Oper. Matrices}, 16(1):95--99, 2022.

\bibitem{VW}
Victor Vinnikov and Hugo~J. Woerdeman.
\newblock Strictly stable hurwitz polynomials and their determinantal
  representations.
\newblock arXiv:2411.17526.

\end{thebibliography}

\end{document}